\documentclass[12pt]{amsart}
\usepackage{times}
\usepackage{enumerate}
\usepackage[T1]{fontenc}
\usepackage{amscd,verbatim,amssymb}
\usepackage[colorlinks,linkcolor=blue,citecolor=blue,urlcolor=red]{hyperref}

\newcommand{\A}{\mathbf{A}}
\newcommand{\C}{\mathbf{C}}
\newcommand{\F}{\mathbf{F}}
\newcommand{\G}{\mathbb{G}}

\renewcommand{\L}{\mathbb{L}}
\newcommand{\sH}{\mathcal{H}}
\newcommand{\N}{\mathbf{N}}
\renewcommand{\P}{\mathbf{P}}
\newcommand{\Q}{\mathbf{Q}}
\newcommand{\sD}{\mathcal{D}}
\newcommand{\sK}{\mathcal{K}}
\newcommand{\sO}{\mathcal{O}}
\newcommand{\sR}{\mathcal{R}}
\newcommand{\sV}{\mathcal{V}}
\newcommand{\sX}{\mathcal{X}}
\newcommand{\bV}{\mathbb{V}}
\newcommand{\Z}{\mathbf{Z}}
\newcommand{\un}{\mathbf{1}}
\newcommand{\Br}{\operatorname{Br}}
\newcommand{\Pic}{\operatorname{Pic}}
\newcommand{\NS}{\operatorname{NS}}
\newcommand{\Chow}{\operatorname{\mathbf{Chow}}}

\newcommand{\DM}{\operatorname{\mathbf{DM}}}
\newcommand{\Spec}{\operatorname{Spec}}
\newcommand{\Ker}{\operatorname{Ker}}
\newcommand{\Coker}{\operatorname{Coker}}
\newcommand{\IM}{\operatorname{Im}}
\newcommand{\Ab}{\operatorname{\mathbf{Ab}}}
\newcommand{\Cor}{\operatorname{\mathbf{Cor}}}
\newcommand{\Hom}{\operatorname{Hom}}

\newcommand{\Sm}{\operatorname{Sm}}
\newcommand{\Reg}{\operatorname{Reg}}

\newcommand{\cd}{\operatorname{cd}}
\newcommand{\ud}{\operatorname{ud}}
\newcommand{\car}{\operatorname{char}}
\newcommand{\trdeg}{\operatorname{trdeg}}
\newcommand{\cont}{{\operatorname{cont}}}
\newcommand{\tors}{{\operatorname{tors}}}

\newcommand{\nr}{{\operatorname{nr}}}
\newcommand{\eff}{{\operatorname{eff}}}
\newcommand{\gm}{{\operatorname{gm}}}
\newcommand{\num}{{\operatorname{num}}}
\newcommand{\cl}{{\operatorname{cl}}}

\newcommand{\et}{{\operatorname{\acute{e}t}}}
\renewcommand{\lim}{\varprojlim}
\newcommand{\colim}{\varinjlim}
\newcommand{\by}{\xrightarrow}
\newcommand{\yb}{\xleftarrow}
\newcommand{\iso}{\by{\sim}}

\newcommand{\inj}{\hookrightarrow}

\newcommand{\surj}{\rightarrow\!\!\!\!\!\rightarrow}
\newcommand{\Surj}{\relbar\joinrel\surj} 
\newcommand{\Tate}{{\operatorname{tate}}}
\renewcommand{\div}{\operatorname{div}}

\newcounter{spec}
\newenvironment{thlist}{\begin{list}{\rm{(\roman{spec})}}%
{\usecounter{spec}\labelwidth=20pt\itemindent=0pt\labelsep=10pt}}%
{\end{list}}

\newtheorem{thm}{Theorem}[section]
\newtheorem{lemma}[thm]{Lemma}
\newtheorem{prop}[thm]{Proposition}
\newtheorem{cor}[thm]{Corollary}
\theoremstyle{remark}
\newtheorem{rque}[thm]{Remark}
\newtheorem{rques}[thm]{Remarks}
\newtheorem{qn}[thm]{Question}

\numberwithin{equation}{section}
\begin{document}

\title[Injectivity of cycle class maps]{On the injectivity and non-injectivity of the $l$-adic cycle class maps}
\author{Bruno Kahn}
\address{CNRS, Sorbonne Université and Université Paris Cité, IMJ-PRG\\ Case 247\\4 place
Jussieu\\75252 Paris Cedex 05\\France}
\email{bruno.kahn@imj-prg.fr}
\begin{abstract}
We study the injectivity of the cycle class map with values in Jannsen's continuous étale cohomology, by using refinements that go through étale motivic cohomology and the ``tame'' version of Jannsen's cohomology. In particular, we use this to show that the Tate and the Beilinson conjectures imply that its kernel is torsion in positive characteristic, and to revisit recent counterexamples to injectivity.
\end{abstract}
\date{
March 10, 2024}
\subjclass[2020]{14C25, 14F20}
\keywords{l-adic cohomology, cycle class map}
\maketitle

\tableofcontents

\section{Introduction}

\subsection{The cycle class map} Recently there has been renewed investigation of Jannsen's cycle class map \cite[(6.14)]{jann}
\begin{equation}\label{eq0}
CH^n(X)\otimes \Z_l\by{\cl^n_X} H^{2n}_\cont(X,\Z_l(n))
\end{equation}
for $X$ a smooth projective variety over a field $k$. Jannsen introduced continuous étale cohomology in loc. cit. in order to correct bad properties of the naïve definition (by inverse limits) of $l$-adic cohomology when $k$ is not separably closed. His definition agrees with the naïve one when the Galois cohomology groups of $k$ with finite coefficients are finite, which happens only for special fields (finite, $p$-adic, separably closed\dots) and not for any infinite finitely generated field.

Finitely generated fields are those for which \eqref{eq0} is especially interesting: Jannsen then proves that it is injective for $n=1$ (loc. cit., Rem. 6.15 a)), and raises the question of such injectivity for higher values of $n$, at least rationally; in \cite[Lemma 2.7]{jann3} he shows that such injectivity would imply the Bloch-Beilinson--Murre (BBM) conjectures on filtrations on (rational) Chow groups. 

One may wonder whether the converse is true: I could not prove it even with the strongest form of the BBM conjectures (existence of a category of mixed motives, \cite[\S 4]{jann3}). One may then wonder whether the injectivity of $\cl^n_X\otimes \Q$ would follow from some other known conjectures. We shall see in in Theorem \ref{p1} c) that the answer is yes in positive characteristic, up to refining $\cl^n_X$ as in \S \ref{s1.2}; to the best of my knowledge, this question remains open in characteristic $0$. See Section \ref{s1} for elementary computations showing that the Bass conjecture is not sufficient, and \S \ref{s5.7}.

Given such an unclear situation, one may wonder whether  the restriction of \eqref{eq0} to the torsion subgroup of $CH^n(X)\otimes \Z_l$ is injective; this has been the topic of \cite[Theorem]{saito}  (resp. \cite{sc-suz,AS22,ct-sc}), where many examples (resp. counterexamples) have been found. 

\subsection{Refining $\cl^n_X$} \label{s1.2} The two ideas developed in this paper are the following:

1) \eqref{eq0} is defined generally for smooth, not necessarily projective, varieties, and this extension can be useful even to study the projective case: it was the central tool in \cite{tatediv} for a simple reduction of the Tate conjecture in codimension $1$ to the case of surfaces over the prime field, and will be used similarly for the proof of Theorem \ref{p1}.

2) As exploited previously in \cite{cycletale}, \eqref{eq0} gains in understanding if it is factored through finer cycle class maps, namely
\begin{multline}\label{eq0bis}
CH^n(X)\otimes \Z_l\by{\alpha^n_X} H^{2n}_\et(X,\Z(n))\otimes \Z_l \\
\by{\beta^n_X}  \tilde H^{2n}_\cont(X,\Z_l(n))\by{\gamma^n_X} H^{2n}_\cont(X,\Z_l(n))
\end{multline}
where $H^{i}_\et(X,\Z(j))$ is étale motivic cohomology and 
\begin{equation}\label{eq1.1}
\tilde H^i_\cont(X,\Z_l(j))=\colim\nolimits_\sX H^i_\cont(\sX,\Z_l(j))
\end{equation} 
 is Jannsen's refinement of $H^i_\cont(X,\Z_l(j))$ from \cite[11.6, 11.7]{jann2}; here $\sX$ runs through models  of $X$ regular, separated and of finite type over $\Z[1/l]$. This group maps to $H^i_\cont(X,\Z_l(j))$, but not isomorphically in general since continuous étale cohomology does not commute with filtering limits of schemes: for example, $\tilde H^1_\cont(\C,\Z_l(1))\simeq \C^*\otimes \Z_l$ while $H^1_\cont(\C,\Z_l(1))=0$; for $l\ne 2$, $\tilde H^2_\cont(\Q,\Z_l(1))\simeq \bigoplus_p \Z_l$ while $H^2_\cont(\Q,\Z_l(1))$ is the $l$-adic completion of $\bigoplus_p \Z$.

The extension of \eqref{eq0} to étale motivic cohomology, i.e. the map $\gamma^n_X\beta^n_X$ of \eqref{eq0bis}, was initially constructed in \cite{glr} using Bloch's cycle complexes; in positive characteristic, this is sufficient to refine it to the middle map $\beta^n_X$, because the models are then defined over a (finite) field. 

In characteristic $0$, the situation would the same if one did not care about the groups $\tilde H^i_\cont(X,\Z_l(n))$, but is more delicate otherwise because of the $\Z$-models $\sX$. A construction of $\cl^n_\sX$ is essentially done in \cite[\S 5]{saito}; namely, Saito defines cycle classes in étale cohomology with finite coefficients, but one can then promote them to continuous étale cohomology as in \cite[Th. 3.23]{jann}; to prove that they pass to rational equivalence, the use of $\P^1$ in the proof of \cite[Lemma 6.14 i)]{jann} can be advantageously replaced by that of $\A^1$, to exploit the $\A^1$-invariance of étale cohomology. 
To define $\beta^n_X$, we need to  use étale motivic cohomology of arithmetic schemes: this is done in Section \ref{s2.1}.

In the sequel, it will be convenient to use the abbreviations 
\begin{equation}\label{eq1.2}
CH^n_\et(X)=H^{2n}_\et(X,\Z(n)), \quad \tilde \cl^n_X =\beta^n_X\alpha^n_X.
\end{equation}

The groups $CH^n_\et(X)$ are in general very amenable to computation, which can be used to simplify and clarify reasonings which did not involve them; we hope that this technique will be taken up by others in future work. 

The factorisation \eqref{eq0bis} also indicates facts on $\cl^n_X$. For example, as an immediate consequence we see that it can be injective on torsion only if $\alpha^n_X$ is. This holds for $n\le 2$ thanks to the short exact sequence \cite[Prop. 2.9]{cycletale}
 \begin{equation}\label{eq9}
 0\to CH^2(X)\by{\alpha^2_X} CH^2_\et(X)\to H^3_\nr(X,\Q/\Z(2))\to 0
 \end{equation}
but not for $n>2$ in general, e.g. \cite[p. 998]{cell}. We shall see in \S \ref{s4} that the counterexamples obtained for $n>2$ are all explained in this way. 

On the other hand, one should not make too much of $CH^n_\et(X)$ beyond being a convenient computational tool: for example, it is far from being finitely generated in general, see Proposition \ref{p3}. Regarding $\tilde H^{2n}_\cont(X,\Z_l(n))$, \eqref{eq0bis} raises the following question:

\begin{qn}[cf. Th. \ref{p1}] Does there exist a finitely generated field $k$, a smooth projective $k$-variety $X$ and an integer $n\ge 0$ such that $\Ker \tilde\cl^n_X\otimes \Q\ne \Ker \cl^n_X\otimes \Q$?
\end{qn}

Another thing which can be done is to find ``obvious'' cases where $\Ker \cl^2_X$ is finite or $0$, by reduction to the case $n=1$. Morally, this might happen when $X$ has a decomposition of the diagonal à la Bloch-Srinivas, because the reduced motive of $X$ has then coniveau $>0$, hence $CH^n(X)$ is approximately $CH^{n-1}(Y)$ for some other smooth projective $Y$, and similarly for continuous étale cohomology. Nevertheless, working out the argument turns out to involve the size of the ground field $k$: see Theorem \ref{t3} and Remark \ref{r1} b).

\subsection{Contents} In Section \ref{stri} we prove some basic facts, the most notable being Theorem \ref{t0}. In Section \ref{s2.1}, we extend \eqref{eq0bis} to other degrees by using motivic cohomology à la Bloch-Levine (see  \eqref{eq00}); the main result is Theorem \ref{t2.1}. In Section \ref{s1} we analyse an attempt to deduce the injectivity of $\cl^n_X\otimes \Q$ over finitely generated fields from the Bass conjecture, showing how it fails. In Section \ref{s2}, we explain that the situation is much better in positive characteristic. In Section \ref{s3}, we prove some results when $X$ has a decomposition of the diagonal: this originates from a letter to Colliot-Thélène of Dec. 1, 2021. 

In Sections \ref{s4}, \ref{s4.4} and \ref{s4.5}, we revisit the counterexamples from \cite{sc-suz}, \cite{AS22} and \cite{ct-sc}. Section \ref{s4} concerns those where $n>2$: we show that they are all explained by non-injectivity of $\alpha^n_X$, so have nothing to do with continuous étale cohomology. In particular, the counterexamples of Alexandrou and Schreieder in \cite{AS22} are proven without using refined Bloch maps. 

The next two sections concern the case $n=2$. In Section \ref{s4.4}, we refine a counterexample of Scavia and Suzuki \cite{sc-suz} involving the Rost motive; in Section \ref{s4.5} we reformulate parts of the paper of Colliot-Thélène and Scavia \cite{ct-sc} in a more concise way. Still, Section \ref{s4} takes one page, Section \ref{s4.4} 1 1/2 pages and Section \ref{s4.5} 4 1/2 pages.

\subsection{Notation} We write $\bV(k)$ for the category of smooth projective varieties over a field $k$, and $\Reg(S)$ (resp. $\Sm(S)$) for the category of regular (resp. smooth) separated schemes of finite type over a base scheme $S$. An \emph{arithmetic scheme} is a connected object of $\Reg(\Spec \Z)$ which is either smooth or not flat over $\Spec \Z$ (hence in the latter case, smooth over a finite field). We recall the following basic fact \cite[II 7.1]{milned}:

\begin{thm}\label{t1.1} Let $\sX$ be an arithmetic scheme. If $F$ is a constructible sheaf on $\sX$ such that $mF=0$ for some integer $m>0$ which is invertible on $\sX$, then the étale cohomology groups $H^i(\sX,F)$ are finite.\qed
\end{thm}

\subsection{Acknowledgements} I thank Marc Levine for his help in the proof of Lemma \ref{l3.1}.

\enlargethispage*{20pt}

\section{Some general facts}\label{stri}

\subsection{Algebra}

\begin{lemma}\label{l3} Let $0\to A\to B\to C\to 0$ be an exact sequence of abelian groups. Assume that the inverse systems $({}_{l^\nu} A)$ and $({}_{l^\nu} B)$ are Mittag-Leffler. Then there is an exact sequence
\[0\to T_l A\to T_l B\to T_l C \to \hat A\to \hat B\]
where $\hat{}$ means $l$-adic completion.
\end{lemma}

\begin{proof} Let $T_l^i$ be the derived functors of $T_l:\Ab \to \Ab$. Write $T_l = \lim \circ U_l$, where $U_l(A)=({}_{l^\nu} A)$: this is a composition of two left exact functors. Since any injective abelian group $I$ is divisible, $U_l(I)$ is Mittag-Leffler hence $\lim$-acyclic, and there is a Grothendieck spectral sequence
\[E_2^{pq} = \lim\nolimits^p U_l^q(A)\Rightarrow T_l^{p+q}(A)\]
for any abelian group $A$. Note that $U_l^1(A) = (A/l^\nu)$ and $\lim^p = U_l^q=0$ for $p,q>1$; thus $T_l^m=0$ for $m>2$, and even for $m=2$ again by Mittag-Leffler. Finally we get a short exact sequence
\[0\to \lim\nolimits^1 {}_{l^\nu} A\to T_l^1 A\to \lim A/l^\nu = \hat A\to 0\]
for any $A\in \Ab$. If $0\to A\to B\to C\to 0$ is an exact sequence, we thus get a commutative diagram of exact sequences:
\[\begin{CD}
& \lim\nolimits^1 {}_{l^\nu} A@>>> \lim\nolimits^1 {}_{l^\nu} B@>>> \lim\nolimits^1 {}_{l^\nu} C\\
& @VVV @VVV @VVV\\
0\to T_l A\to T_l B\to T_l C \to & T_l^1 A @>>> T_l^1 B@>>> T_l^1 C\\
& @VVV @VVV @VVV\\
& \hat A@>>> \hat B@>>> \hat C.
\end{CD}\]

Under the hypothesis of the lemma, the $\lim^1$ vanish, hence the conclusion.
\end{proof}

\enlargethispage*{20pt}

Let $\Ab$ be the category of abelian groups; as in previous works, we  shall use the category $\Ab\otimes \Q$ of abelian groups up to isogenies, i.e. the  localisation of $\Ab$ by the Serre subcategory of groups of finite exponent.

\begin{lemma}\label{l1} a) For an abelian group $A$, the following conditions are equivalent:
\begin{thlist}
\item $A$ is the direct sum of a free finitely generated group and a group of finite exponent.
\item $A_\tors$ is of finite exponent and $\bar A:=A/A_\tors$ is free finitely generated.
\item $A$ is a sum of a finitely generated subgroup and a subgroup of finite exponent.
\item The image of $A$ in $\Ab\otimes \Q$ is isomorphic to a finitely generated group.
\end{thlist}
We say that such $A$ is \emph{finitely generated modulo isogenies} (in short: fgmi).\\
b) The fgmi abelian groups form a Serre subcategory of $\Ab$.
\end{lemma}

\begin{proof} a) (i) $\iff$ (ii) $\Rightarrow$ (iii) $\Rightarrow$ (iv)  are trivial. (iv) $\Rightarrow$ (i): let $f:A\iso M$ be an isomorphism in $\Ab\otimes \Q$, where $M$ is finitely generated. If $\bar M$ is the quotient of $M$ by its torsion subgroup, then $M\to \bar M$ is an isomorphism in $\Ab\otimes \Q$; thus we may assume $M$ torsion-free, i.e. free. By calculus of fractions, there is a diagram in $\Ab$:
\[M\yb{u} \tilde M \by{\tilde f} A\]
where the kernel and cokernel of $u$ and $f$ are of finite exponent, and such that $f=\tilde f u^{-1}$ in $\Ab\otimes \Q$. Replacing $M$ by $\IM u$, we may assume $u$ surjective, hence split by some homomorphism $v:M\to \tilde M$. Then $\Ker fv$ and $\Coker fv$ have finite exponent; hence the composition
\[M\by{fv} A\by{\pi} \bar A\]
is injective and its cokernel has finite exponent, say $N$. Then $N\bar A\subseteq M$ is free finitely generated, hence so is $\bar A$. But then $\pi$ is split and $A\simeq \bar A\oplus A_\tors$. Finally, since $\Coker fv$ has finite exponent, so does $A_\tors$.

b) Let $0\to A'\to A\to A''\to 0$ be an exact sequence in $\Ab$. We must show that $A$ is fgmi if and only if so are $A'$ and $A''$. Suppose that $A$ is fgmi. By Condition (iii), so is $A''$. Moreover, $A'_\tors\inj A_\tors$ and $\bar A'\inj \bar A$, hence so is $A'$ as well by Condition (ii). Suppose now that $A'$ and $A''$ are fgmi. The exact sequence
\[0\to A'_\tors \to A_\tors \to A''_\tors\]
shows that $A_\tors$ has finite exponent. Moreover, if $N$ is an exponent of $A''_\tors$, the image of the map
\[\bar A'\to \Ker(\bar A\to \bar A'')\]
contains $N\Ker(\bar A\to \bar A'')$. Therefore $\Ker(\bar A\to \bar A'')$ is free finitely generated; since so is $\bar A''$, the surjection $\bar A\to \bar A''$ is split and $\bar A$ is also free finitely generated. We conclude with Condition (ii) again.
\end{proof}

\subsection{Geometry}

For a scheme $X$, write $\cd_l(X)$ for its étale $l$-cohomological dimension and  $\hat{\cd}_l(X)$ for the inf of those $i\ge 0$ such that $H^i_\cont(X,(F_\nu))=0$ for all inverse systems $(F_\nu)$ of sheaves of $l$-primary torsion on $X_\et$. The following lemma puts \cite[Prop. 2.3]{ct-sc} in its right generality (same proof).

\begin{lemma}\label{l4} $\hat{\cd}_l(X)=\cd_l(X)$.\qed
\end{lemma}

We also note the following result, whose proof is identical to that of \cite[Th. 5.1]{saito} and which implies it:

\begin{thm}\label{t0} Let $X\in \Sm(\Spec k)$, where $k$ is a finitely generated field of characteristic $\ne l$. For any $n\ge 0$, the image of $\tilde\cl^n_X$, and a fortiori that of $\cl^n_X$, is a finitely generated $\Z_l$-module.\qed
\end{thm}

(This proof is clarified if one promotes Saito's argument to a commutative diagram
\[\begin{CD}
CH^n(\sX)\otimes \Z_l&\Surj& CH^n(X)\otimes \Z_l\\
@V\cl^n_\sX VV @V\tilde \cl^n_X VV\\
H^{2n}_\cont(\sX,\Z_l(n))@>>> \tilde H^{2n}_\cont(X,\Z_l(n))
\end{CD}\]
where $\sX$ is an arithmetic model of $X$, using the finite generation of $H^{2n}_\cont(\sX,\Z_l(n))$ which follows from Theorem \ref{t1.1}.)

\section{Motivic cohomology}\label{s2.1} 

\subsection{Refined motivic cycle class maps}\label{s3.1} In the sequel we shall use a generalisation of \eqref{eq0bis} to all motivic cohomology:
\begin{multline}\label{eq00}
H^i(X,\Z(n))\otimes \Z_l\by{\alpha^{n,i}_X} H^i_\et(X,\Z(n))\otimes \Z_l\\
\by{\beta^{n,i}_X} \tilde H^i_\cont(X,\Z_l(n))\by{\gamma^{n,i}_X} H^i_\cont(X,\Z_l(n)).
\end{multline}

This was done in \cite[\S\S 2 and 3A]{cycletale}, except for $\tilde H^i_\cont(X,\Z_l(n))$. The same comments as in \S \ref{s1.2} apply,  except that one should explain which version of motivic cohomology is used for arithmetic schemes. We use Levine's extension of Bloch's complexes to schemes over a Dedekind domain \cite{levine}, developed by Geisser in \cite{geisser3}: more precisely, the Zariski (resp. étale) hypercohomology of Bloch's cycle complexes, namely the complexes of Zariski sheaves
\[\Z(n) := z^n(-,*)[-2n]\] 
and their étale versions $\Z(n)_\et$, as in \cite[\S 3]{geisser3}.

To define $\beta^{n,i}_\sX$ for an arithmetic scheme $\sX$, we use the isomorphism
\begin{equation}\label{eq8}
\Z(n)_\et\otimes \Z/l^\nu\simeq \mu_{l^\nu}^{\otimes n};
\end{equation}
extended to arithmetic schemes in \cite[Th. 1.2 4]{geisser3}. Since the étale cohomology of $\sX$ with finite coefficients is finite (Theorem \ref{t1.1}), $\beta^{n,i}_\sX$ may simply be seen as $l$-adic completion. To pass to smooth schemes over a field, we use the continuity of motivic and étale motivic cohomology (commutation with filtering inverse limits of schemes with affine transition morphisms): this follows from
\begin{itemize}
\item continuity of the cycle complexes themselves, 
\item continuity of Zariski and étale cohomology.
\end{itemize}

At the referee's request, we give details. The first point is seen explicitly from the definition of the terms of $\Z(n)$: for any scheme $\sX$ over a Dedekind scheme $S$, one has by definition
\[z^n(\sX,i) = \bigoplus_Z \Z\]
where $Z$ runs through the integral closed subschemes of codimension $n$ of $\sX\times_S\Delta^i_S$ which meet all faces properly.

In the second point, continuity for sheaves is classical: see \cite[Th. 5.7]{sga4} for étale cohomology. We reduce to this case by a hypercohomology spectral sequence argument. Details on the delicate points to deal with complexes of sheaves which are not [known to be] bounded below, and how to solve them, are given in \cite[\S 2C]{cycletale} in the case of schemes over a field; here one can proceed exactly in the same way by using \eqref{eq8} and \cite[Prop. 3.6]{geisser3}.

\begin{lemma}\label{l2} Let $X\in \Sm(k)$, where $k$ is a field. In \eqref{eq00}, $\beta^{n,i}_X$, $\gamma^{n,i}_X$ and $\gamma^{n,i}_X\beta^{n,i}_X$ have divisible kernels and torsion-free cokernels, while $\alpha^{n,i}_X$ has torsion kernel and cokernel. Moreover, $\alpha^{n,2n}_X=\alpha^n_X$ is bijective for $n=1$ and injective for $n= 2$.
\end{lemma}

\begin{proof} For $\gamma^{n,i}_X\beta^{n,i}_X$, see \cite[Cor. 3.5]{cycletale}; the proof is the same for $\beta^{n,i}_X$ and $\gamma^{n,i}_X$. For $\alpha^{n,i}_X$, see \cite[Th. 2.6 c)]{cycletale}. The last claim follows from the isomorphism
\begin{equation}\label{gm}
\Z(1)\simeq \G_m[-1].
\end{equation} 
for $n=1$ and from \eqref{eq9} for $n=2$.
\end{proof}

Note that if $\cl^n_X$ is not injective, neither is $\beta^n_X$ in \eqref{eq0bis}, and then $CH^n_\et(X)$ is \emph{not} finitely generated since it contains a nonzero divisible subgroup; see Proposition \ref{p3} for (minimal) examples and Theorem \ref{l5.1} for positive cases over a finite field.

\subsection{Naïve and non naïve higher Chow groups} For an arithmetic scheme $\sX$, write $H^*(\sX,\Z(n))_n$ for the cohomology groups of the Bloch-Levine cycle complex $z^n(\sX,*)$: more precisely, $H^i(\sX,\Z(n))_n \allowbreak= H^{2n-i}(z^n(\sX,*))$. These are the \emph{naïve higher Chow groups}. They map to motivic cohomology, but this map is mysterious in general. Nevertheless:

\begin{thm}\label{t2.1} The map $H^i(\sX,\Z(n))_n\to H^i(\sX,\Z(n))$ is bijective for $i\ge 2n$; in particular, $CH^n(\sX)\iso H^{2n}(\sX,\Z(n))$.
\end{thm}

This will be used in the proof of Theorem \ref{t3}.

\begin{proof} If $\sX$ were smooth over a field, this would follow from Bloch's localisation theorem for naïve higher Chow groups: the latter is extended in \cite{levine} only to smooth schemes over a semi-local Dedekind ring. The strategy is to reduce to the field case.

Let $X$ be the generic fibre of $\sX$ over $\Spec \Z$, and for each prime $p$ $\sX_p$ be its fibre at $p$ (it may be empty). By \cite[Th. 1.7]{levine}, we have a long exact sequence
\begin{multline}\label{eq2.1}
\dots\to H^{i-1}(X,\Z(n))\to \bigoplus_p H^{i-2}(\sX_p,\Z(n-1))\\
\to H^i(\sX,\Z(n))\to H^i(X,\Z(n))\to \dots
\end{multline}

Let $i>2n$. The claim is that $H^i(\sX,\Z(n))=0$: this follows from \eqref{eq2.1}, since $H^i(X,\Z(n))=H^{i-2}(\sX_p,\Z(n-1))=0$.

For $i=2n$, we compare \eqref{eq2.1} with an exact sequence for (ordinary) Chow groups. Namely, we have the following commutative diagram of exact sequences
\begin{equation}\label{eq2.2}\tiny
\begin{CD}
A^{n-1}(X,K_n^M)&\by{\delta}& \bigoplus\limits_p CH^{n-1}(\sX_p) &\to& CH^{n}(\sX)&\surj& CH^{n}(X)\\
&& @Vb V\wr V @Vc VV @Vd V\wr V\\
H^{2n-1}(X,\Z(n))&\by{\delta'}& \bigoplus\limits_p H^{2n-2}(\sX_p,\Z(n-1)) &\to& H^{2n}(\sX,\Z(n))&\to& H^{2n}(X,\Z(n)).
\end{CD}
\end{equation}

Here, $b$ and $d$ are bijective because we are over fields, and the top exact sequence is the one of Fulton \cite[Prop. 1.8]{fulton}, extended to the left thanks to the Gersten complexes associated to $\sX$, $X$ and $\sX_p$ as in Rost \cite[p. 356]{rost}.  This already gives the surjectivity of $c$, and its injectivity follows from

\begin{lemma}\label{l3.1}  
Let $h^n(X,1)=\Ker(z^n(X,1)\by{d'} z^n(X,0))$, where $d'$ is the differential of the complex $z^n(X,*)$. There exists a homomorphism $\Theta:h^{n}(X,1)\to  A^{n-1}(X,K_n^M)$ such that $\delta'\pi =b\delta\circ \Theta$, where $\pi:h^n(X,1)\to H^{2n-1}(X,\Z(n))$ is the natural surjection.
\end{lemma}

\begin{proof} In order to motivate the construction, we first review the definition of an isomorphism $\theta: H^{2n-1}(X,\Z(n))\iso A^{n-1}(X,K_n^M)$: 
as we are over a field, we only need to use Bloch's results.  
The localisation theorem of \cite[Cor. (0.2)]{bloch-corr} implies Gersten's conjecture for the cohomology sheaves $\sH^q(\Z(n))$ \cite[Th. (10.1)]{bloch-adv}. This in turn implies that $\Z(n)$ is acyclic in degrees $>n$
and, by the theorem of Nesterenko-Suslin and Totaro \cite[Th. 4.9]{ns}, \cite{totaromilnor}, we get an isomorphism of $\sH^n(\Z(n))$ with the $n$-th unramified Milnor $K$-sheaf $\sK_n^M$ of \cite[p. 360]{rost} (see \cite[p. 144]{ns}). Thus we get an exact triangle
\[\tau_{<n}\Z(n)\to \Z(n)\to \sK_n^M[-n]\by{+1}\]
which yields a long exact sequence of hypercohomology groups
\begin{multline*}
\dots H^{2n-1}(X,\tau_{<n}\Z(n))\to H^{2n-1}(X,\Z(n))\by{\theta} H^{n-1}(X,\sK_n^M)\\
\to H^{2n}(X,\tau_{<n}\Z(n))\to \dots
\end{multline*}
and $H^{n-1}(X,\sK_n^M)$ is canonically isomorphic to $A^{n-1}(X,K_n^M)$ by Gersten's conjecture (cf. \cite[Cor. 6.5]{rost}). Finally, Gersten's conjecture also implies that $H^p(X,\sH^q(\Z(n)))=0$ for $p>q$, hence $H^m(X,\tau_{<n}\Z(n))=0$ for $m>2n-2$ by the hypercohomology spectral sequence.

Ideally, one should then show that $b\delta\theta=\delta'$. For this, we would need a description of $\theta$ on the chain level. Instead we shall only prove the statement of the lemma, using the following construction due to Marc Levine.

Define a map $\Theta_0:z^n(X,1)\to C^{n-1}(X,K_n^M)$ as follows. We identify $\Delta^1$ with $\A^1$ by sending $(0,1)$ to $0$ and $(1,0)$ to $1$. It suffices to define $\Theta_0$ on the integral generators $W$ of $z^n(X,1)\subset Z^n(X\times \A^1)$. Let $Z\subset X$ be the closure of $p_1(W)$. If $W\to Z$ is not generically finite, we set $\Theta_0(W)=0$. In this case, $W=Z\times \A^1$ is the boundary of $Z \times \Delta^2$, so we can neglect these cycles. If $W\to Z$ is generically finite, the function $t/(t-1)$ on $\A^1$ restricts to a function $f$ on $W$, and $W(0)-W(1)= \div(g)$ where $g=N_{k(W)/k(Z)}(f)$ \cite[Prop. 1.4 and \S 1.6]{fulton}. We set $\Theta_0(W)=(Z,g)\in \bigoplus_{x\in X^{(n-1)}} k(x)^*$. By construction, the diagram
\begin{equation}\label{eq3.1}
\begin{CD}
C^{n-1}(X,K_n^M)@>d>> C^n(X,K_n^M)=Z^n(X)\\
@A\Theta_0 AA @A||AA\\
z^n(X,1)@>d'>> z^n(X,0) = Z^n(X)
\end{CD}
\end{equation}
commutes. In particular $\Theta_0$ sends $\Ker d'$ to $\Ker d$, hence defines a map $\Theta$ as in the lemma.

It is likely that $\Theta$ represents $\theta$, but this is not necessary: we only need the statement of the lemma to conclude the proof of Theorem \ref{t2.1}.  But we can repeat the same construction verbatim by replacing $X$ with $\sX$; we get a commutative diagram  similar to \eqref{eq3.1}, which receives a map from \eqref{eq3.1} by the operation ``closing a cycle of $X$ in $\sX$''. For both localisation sequences: the one involving Rost cycle complexes and the one involving Bloch-Levine cycle complexes, the boundary map $\delta$ (resp. $\delta'$) is obtained by closing up a cycle on $X$ in $\sX$ and then applying the differential of the complex. (In the second case, the point is that there are no cycles to move, so the construction can done globally on $\Spec \Z$.) This shows that, if $w\in h^n(X,1)$, then the boundary of (the cohomology class of) $w$ coincides with that of $\Theta(w)$.
\end{proof}

This concludes the proof of Theorem \ref{t2.1}.
\end{proof}

\enlargethispage*{20pt}

\section{A naïve calculation} \label{s1}

Here we mimick for $n>1$ the argument giving injectivity of $\cl^n_X$ for $n=1$ by using étale motivic cohomology, and see what goes wrong.

The isomorphism \eqref{eq8} yields short exact sequences
\begin{equation}\label{eq10}
0\to H^j_\et(X,\Z(n))/l^\nu\to H^j_\et(X,\mu_{l^\nu}^{\otimes n})\to {}_{l^\nu} H^{j+1}_\et(X,\Z(n))\to 0.
\end{equation}

Taking the inverse limit and using Mittag-Leffler, we get a short exact sequence
\[0\to H^j_\et(X,\Z(n))^{\wedge}\by{b} \lim H^j_\et(X,\mu_{l^\nu}^{\otimes n})\to T_l (H^{j+1}_\et(X,\Z(n)))\to 0\]
where $^{\wedge}$ means $l$-adic completion. For $j=2n$, this yields a commutative diagram, with notation as in  \eqref{eq0bis}:
\[\begin{CD}
CH^n(X)\otimes \Z_l@>\alpha^n_X>> CH^{n}_\et(X)\otimes \Z_l @>\gamma^n_X\beta^n_X>>  H^{2n}_\cont(X,\Z_l(n))\\
@Vc VV@Vc_\et VV @VdVV\\
CH^n(X)^{\wedge}@>\hat\alpha^n_X>> CH^{n}_\et(X)^{\wedge}@>b>> \lim H^{2n}_\et(X,\mu_{l^\nu}^{\otimes n})
\end{CD}\]
where  $b$ is injective. This shows that $\Ker \cl^n_X\subseteq \Ker (\hat \alpha^n_X c)$.

 If $n=1$, $\alpha^1_X$ is an isomorphism, hence so is $\hat \alpha^1_X$. By hypothesis, $CH^1(X)$ is finitely generated \cite{picfini}, hence $c$ is an isomorphism and $\cl^1_X$ is injective,  cf. \cite[Rem. 6.15 a)]{jann}. Suppose now that $n\ge 2$. If $CH^n(X)$ is finitely generated (Bass conjecture), $c$ is an isomorphism and $\Ker \alpha^n_X$ is finite (Lemma \ref{l2}). But the same is far from clear for $\Ker \hat \alpha^n_X$: for example, nothing prevents a priori $CH^{n}_\et(X)$ from being $l$-divisible, hence $CH^{n}_\et(X)^{\wedge}$ from being $0$!
 
 We can approach this latter kernel via Lemma \ref{l3}, neglecting the finite group $\Ker \alpha^n_X$.  If $X$ is of finite type over $\Spec \Z[1/l]$, the groups $H^{2n-1}_\et(X,\mu_{l^\nu}^{\otimes n})$ are finite and the Mittag-Leffler hypotheses of this lemma are verified; hence we get an exact sequence, up to a finite group
\[T_l (CH^{n}_\et(X))\to T_l (\Coker \alpha^n_X) \to \Ker \hat\alpha^n_X\to 0.  \]

 For $n=2$, we are contemplating the effect of $T_l$ on the map on the right of \eqref{eq9}. It seems difficult to get further in general with such an approach.

Finally, one might hope to reason directly with $CH^n_\et(X)$ rather than $CH^n(X)$. However this approach is doomed: 

\begin{prop}\label{p3}
The map $\beta^2_X$ is not injective in general  in \eqref{eq0bis}, and $CH^2_\et(\sX)$ is in general not finitely generated for arithmetic schemes $\sX$.
\end{prop}

\begin{proof}
For the first point, simply take $X=\Spec k$ with $\cd_l(k)=3$: on the one hand we have an isomorphism 
\[0\ne H^3(k,\Q_l/\Z_l(2))\iso CH^2_\et(k)\otimes \Z_l\] 
(this non-vanishing is classically proven by using two successive discrete valuations of rank $1$ to descend to the cohomology of a finite field, see \cite[Rem. 6.4]{sc-suz}). On the other hand,   $\tilde H^4_\cont(k,\Z_l(2))=0$  by Lemma \ref{l4} because a cofinal system of models of $\Spec k$ over $\Spec \Z[1/l]$ has cohomological dimension $3$, so $\beta^2_X$ is not injective. 

The same holds for any such model: for example, taking $\sX=\sX_0\times \G_m$ with $\sX_0$ either the spectrum of a ring of $S$-integers in a number field (with $S$ containing all places above $l$), or a smooth \emph{affine} curve over $\F_p$ (with $p\ne l$), the Gysin exact sequence and $\A^1$-invariance of étale motivic cohomology away from $p$ give an isomorphism
\[CH^2_\et(\sX)\otimes\Z_{(l)}\iso H^3_\et(\sX_0,\Z_{(l)}(1)) = \Br(\sX_0)\otimes \Z_{(l)}\]
where $\Br(\sX_0)$ is the Brauer group of $\sX_0$. Here again, $\gamma^2_\sX=0$ and $CH^2_\et(\sX)\{l\}\allowbreak=CH^2_\et(\sX)\otimes\Z_{(l)}$ is nonzero and divisible.
\end{proof}

The first example in this proof seems basic; it already appears in \cite{sc-suz} and \cite{ct-sc}, and will reappear in Sections \ref{s4.4} and \ref{s4.5}.

\section{The case of positive characteristic}\label{s2}

\subsection{The Tate-Beilinson conjecture} We start with a smooth projective variety $X$ over $k=\F_p$. Let $\bar k$ be an algebraic closure of $k$, $G=Gal(\bar k/k)$ and $X^s=X\otimes_k \bar k$. There are two basic conjectures:

\begin{description}
\item[Tate's conjecture] For all $n\ge 0$, the order of the pole of the zeta function $\zeta(X,s)$ at $s=n$ is equal to the rank of $A^n_\num(X)$, the group of cycles of codimension $n$ on $X$ modulo numerical equivalence.
\item[Beilinson's conjecture] rational and numerical equivalences agree on $X$ (with rational coefficients).
\end{description}

We call these two conjectures, taken together, the \emph{Tate-Beilinson conjecture}.

Of course, Beilinson's conjecture implies the injectivity of $\cl_X^*\otimes \Q$ (because it implies that rational and homological equivalences agree). We shall recall in \S \ref{s2.2} that the Tate-Beilinson conjecture extends this injectivity to \emph{open} $X$'s. 

Tate's conjecture taken alone implies:
\begin{description}
\item[Cohomological Tate conjecture] the cycle class map $CH^n(X)\otimes \Q_l\allowbreak\to H^{2n}(\bar X,\Q_l(n))^G$ is surjective for all $n\ge 0$.
\end{description}

Conversely, the cohomological Tate conjecture implies Tate's conjecture in the presence of the Grothendieck-Serre conjecture (semi-simplicity of the action of $G$ on the cohomology of $\bar X$), \cite[Th. 2.9]{tate}.

Since $\cd_l(G)=1$, one has short exact sequences
\[0\to H^{2n-1}(\bar X,\Q_l(n))_G\to H^{2n}_\cont(X,\Q_l(n))\to H^{2n}(\bar X,\Q_l(n))^G\to 0.\]

But the left hand side is $0$ by (Deligne's proof of) the Weil conjectures (the Frobenius eigenvalues are Weil numbers of weight $2n-1-2n=-1$), so the cohomological Tate conjecture is equivalent to the \emph{surjectivity} of $\cl^n_X\otimes \Q$.

\subsection{Function fields} Suppose now that $V$ is a smooth variety over a function field $K/\F_p$. The aim of this section is to prove: 

\begin{thm}\label{p1} The Tate-Beilinson conjecture implies that\\
a) $CH^n(V)$ is fgmi (see Lemma \ref{l1}).\\
b) The kernel of the refined cycle class map
\[\beta^n_V:CH^{n}_\et(V)\otimes \Z_l\to \tilde H^{2n}_\cont(V,\Z_l(n))\]
is torsion (and divisible).\\
c) The kernel of the refined cycle class map
\[\tilde \cl^n_V:CH^n(V)\otimes \Z_l\to \tilde H^{2n}_\cont(V,\Z_l(n))\]
has finite exponent.
\end{thm}

For the proof, see \S \ref{s5.6}.

As an application, assume $V$ affine and let $n=\dim V$. By the cohomological dimension of affine schemes, $\tilde H^a_\cont(K,H^b(\bar V,\Z_l(n)))=0$ for $b>n$. If moreover $\trdeg(K/\F_p)<n-1$, the same holds for $a\ge n$, hence $\tilde H^{2n}_\cont(V,\Z_l(n))=0$ by the Hochschild-Serre spectral sequence of \cite[Th. 3.3]{jann} and

\begin{cor}\label{c1} Under the above hypotheses, $CH_0(V)$ is a group of finite exponent.\qed
\end{cor}

\subsection{Nilpotence} Let $N(k)\subseteq \bV(k)$ be the class of varieties $X$ such that the ideal of Chow self-correspondences 
\[\Ker(CH^{\dim X}(X\times X)_\Q\to A^{\dim X}_\num(X\times X)_\Q)\]
is nil, where $A^{\dim X}_\num(X\times X)$ denotes cycles modulo numerical equivalence. The following is a version of the main result of \cite{cell}:

\begin{thm}\label{t1} Let $X\in N(k)$. If $X$ satisfies Tate's conjecture, it also satisfies Beilinson's conjecture.
\end{thm}

(Conversely, if $X$ verifies Beilinson's conjecture then obviously $X\in N(k)$.)

Theorem \ref{t1} is a consequence of the following more precise theorem, in which we use Voevodsky's triangulated category of geometric motives $\DM_\gm(k)$ \cite{voetri} to interpret motivic cohomology as Hom groups. It refines \cite[Prop. 10.5.1]{ihes}, under stronger hypotheses.

\begin{thm}\label{t2} a) If $X\in N(k)$ verifies Tate's conjecture, its motivic cohomology groups $H^i(X,\Z(n))$ are all fgmi. Moreover, they are torsion unless $i=2n$, in which case the projection of $H^{2n}(X,\Z(n))\simeq CH^n(X)$ onto $A^n_\num(X)$ is the projection on its maximal torsion-free quotient. Here, $A^n_\num(X)$ denotes cycles modulo numerical equivalence.\\
b) If $N(k)=\bV(k)$ and all $X\in \bV(k)$ verify Tate's conjecture, then the $H^i(U,\Z(n))$ are fgmi for all smooth $k$-varieties $U$ and all $n\ge 0$, $i\in\Z$. More generally, all Hom groups in $\DM_\gm(k)$ are fgmi, and the pairings
\[\Hom(\Z,M)\times \Hom(M,\Z)\to \Hom(\Z,\Z)=\Z\]
are perfect in $\Ab\otimes \Q$ (i.e. the induced homomorphisms 
\[\Hom(\Z,M)\to \Hom(M,\Z)^*\] 
are isomorphisms in this category).
\end{thm}

\begin{proof} As in \cite{cell}, it is an elaboration of an argument going back to Soulé \cite{soule} and Geisser \cite{geisser}.

The various absolute Frobenius endomorphisms assemble to yield a $\otimes$-endomorphism $F$ of the identity functor of $\DM_\gm(k)$. Namely, finite correspondences clearly commute with Frobenius endomorphisms; this extends them to $\Cor(k)$, then to its bounded homotopy category, etc. Since $F_{U\times V} = F_U\times F_V$, this endomorphism is indeed monoidal. Moreover, it is compatible with the Frobenius action on Chow motives via the $\otimes$-functor
\begin{equation}\label{eqchow}
\Phi:\Chow(k)\to \DM_\gm(k)
\end{equation}
of \cite[Prop. 2.1.4]{voetri}. This allows us to compute $F_{\Z(n)[i]}$ for any $n,i$: first
\[F_{\Z(n)[i]}=F_{\Z(n)}=F_{\Z(1)}^{\otimes n}\]
then
\[F_{\Z(1)} = F_{\Z(1)[2]} = \Phi(F_\L)\] 
where $\L$ is the Lefschetz motive, and it is known that $F_\L=p$ (this is also obvious from \eqref{gm}). Finally we get:
\[F_{\Z(n)[i]}=p^n.\]

Therefore, for any smooth $U$ and any $n\ge 0$,  $i\in \Z$, we have for any $\phi\in H^i(U,\Z(n))=\Hom(M(U),\Z(n)[i]) $:
\begin{equation}\label{eq5}
F_U^*\phi=\phi\circ F_U  = F_{\Z(n)[i]} \circ \phi = p^n\phi. 
\end{equation}

a) Let $X\in \bV(k)$. Write the (rational) numerical motive $h_\num(X)$ as a direct sum of simple motives $\bigoplus_\alpha S_\alpha$, thanks to Jannsen's theorem \cite{jannsen}. If $X\in N(k)$, we can lift this decomposition to a decomposition of the Chow motive of $X$ (see \cite[Lemma 5.4]{jann3}):
\[h(X)\simeq\bigoplus_\alpha \tilde S_\alpha.\]

Moreover, if $F_\alpha$ is the Frobenius endomorphism of $S_\alpha$ and if $P_\alpha$ is its minimal polynomial, there exists $N_\alpha>0$ such that $P_\alpha(F_\alpha)^{N_\alpha}=0$.

In $\DM_\gm(k)\otimes \Q$, we therefore have a decomposition
\[M(X)\simeq \bigoplus_\alpha \Phi(\tilde S_\alpha)\]
hence a decomposition in $\Ab\otimes \Q$
\[\widetilde{\Hom}(M(X),\Z(n)[i])\simeq \bigoplus_\alpha\widetilde{\Hom}(\Phi(\tilde S_\alpha),\Z(n)[i])\]
where $\widetilde{\Hom}$ are the refined Hom groups of  \cite[Rem. 4.13]{cycletale}). 
By \eqref{eq5}, we therefore have
\[P_\alpha(p^n)^{N_\alpha} \widetilde{\Hom}(\Phi(\tilde S_\alpha),\Z(n)[i])=0\; \forall\alpha\]
hence
\[\widetilde{\Hom}(\Phi(\tilde S_\alpha),\Z(n)[i])=0 \text{ if }P_\alpha(p^n)\ne 0.\]

Suppose moreover that $X$ verifies the Tate conjecture. By \cite[Prop. 2.6]{milne} (see also \cite[Th. 2.7]{geisser}), $P_\alpha(p^n)=0$ $\iff$ $S_\alpha\simeq \L^n$. Thus we get
\[ \widetilde{\Hom}(M(X),\Z(n)[i])\simeq
\begin{cases}
0&\text{if } i\ne 2n\\
A^n_\num(X)&\text{if } i=2n
\end{cases}
\]
which is equivalent to a).

In b), the first statement is a special case of the second; thanks to Lemma \ref{l1} b), to get it we reduce to a) by de Jong's theorem \cite{dJ}, observing that the statement of c) (taken for $M[i]$, all $i\in\Z$) is stable under cones and direct summands.
\end{proof}

\subsection{Examples}\label{s2.3} As in \cite[Def. 1]{cell}, let $B(k)$ denote the class of $X\in \bV(k)$ whose Chow motive $h(X)$ is a direct summand of $h(A_E)$, where $A$ is an abelian variety over $k$ and $E$ is a finite extension of $k$; and let $B_\Tate(k)$ denotes the class of members of $B(k)$ which verify the cohomological Tate conjecture. We have
\[B_\Tate(k)\subseteq B(k)\subseteq N(k)\subseteq \bV(k)\]
where the second inclusion follows from the nilpotence theorem of Kimura-O'Sullivan \cite[Prop. 7.5]{kim}. Since the  Galois action on cohomology is semi-simple for any $X\in B(k)$ \cite[Lemme 1.9]{cell}, members of $B_\Tate(k)$ also verify the strong form of the Tate conjecture and Theorem \ref{t2} a) thus applies to them. 

Conversely, the Tate conjecture implies that every smooth projective variety is of abelian type modulo \emph{numerical} equivalence \cite[Rem. 2.7]{milne}, hence modulo \emph{rational} equivalence under the Beilinson conjecture; therefore the Tate-Beilinson conjecture implies that $B_\Tate(k)=\bV(k)$.

All varieties known to belong to $N(k)$ actually belong to $B(k)$.

\subsection{Weil-étale cohomology}\label{s2.2} In positive characteristic, there is another cohomology introduced by Lichtenbaum for  $X\in \Sm(k)$: \emph{Weil-étale cohomology}. It gives rise to \emph{Weil-étale motivic cohomology}, that we denote by $H^*_W(X,\Z(n))$. By \cite[Th. 6.1]{geisser2}, there are long exact sequences
\begin{multline}\label{e1b}
\dots\by{\partial} H^i_\et(X,\Z(n))\to H^i_W(X,\Z(n))\\
\to H^{i-1}_\et(X,\Z(n))\by{\partial}H^{i+1}_\et(X,\Z(n))\to\dots
\end{multline}
where $\partial$ is given by the composition
\begin{multline*}
H^{i-1}_\et(X,\Z(n))\otimes \Q=H^{i-1}_\et(X,\Q(n))\to
H^{i-1}_\et(X,\Q/\Z(n))\overset{\cdot e}{\to}\\ 
H^{i}_\et(X,\Q/\Z(n))\overset{\beta}{\to}H^{i+1}_\et(X,\Z(n))
\end{multline*}
in which $e$ is the canonical generator of $H^1_\cont(\F_p,\Z_l)=\Hom_\cont(G,\Z_l)$ given by the arithmetic Frobenius.
They rely on the similar exact sequences of  \cite[Prop. 4.2]{sheaf}, plus  \cite[Th. 5.1 and 5.3]{geisser2}. (See \cite[App. A]{ihes} for a different proof of the latter.)

The étale cycle class map $\beta^{n,i}_X$ extends to a Weil-étale cycle class map
\begin{equation}\label{eq2b}
H^i_W(X,\Z(n))\otimes \Z_l\to H^i_\cont(X,\Z_l(n))
\end{equation}
and

\begin{thm}\label{l5.1} The Tate-Beilinson conjecture implies that, for all $(i,n)\in\Z\times \N$,
\begin{thlist}
 \item $H^i_W(X,\Z(n))$ is finitely generated for all $X\in \bV(k)$;
\item $H^i_\et(X,\Z(n))\iso H^i_W(X,\Z(n))$ if $i\le 2n$ for all $X\in \bV(k)$;
 \item \eqref{eq2b} is an isomorphism for all $X\in \Sm(k)$.
\end{thlist}
\end{thm}

For $i=2n$, (ii) together with \eqref{e1b} predicts a precise description of the $l$-adic cycle class map and in particular of its kernel, which is seen to be \emph{torsion}.

\begin{proof} These facts are proven unconditionally in \cite[Th. 3.6 and Cor. 3.8]{cell} for all  $X\in B_\Tate(k)$; but the Tate-Beilinson conjecture implies that $B_\Tate(k)=\bV(k)$ as seen in \S \ref{s2.3}. For the open case in (iii), one proceeds by dévissage from the projective case as in the proof of \cite[Lemma 3.8]{glr}.
\end{proof}

\begin{rque} Conversely, items (i) and (iii) of Theorem \ref{l5.1} imply the Tate-Beilinson conjecture, see \cite[Th. 8.4]{geisser2}; this will not be used here.
\end{rque}

\subsection{Proof of Theorem \ref{p1}} \label{s5.6}

\begin{proof} We can choose a smooth model $S$ of $K$, and (up to shrinking $S$) a smooth model $U\to S$ of $V$. Since $CH^n(V)$ is a quotient of $CH^n(U)$, a) follows from Theorem \ref{t2} b) and Lemma \ref{l1} b). b) follows from \eqref{e1b} and the isomorphism \eqref{eq2b} (Theorem \ref{l5.1} (iii)) by passing to the limit over $U$. Finally, c) follows from a) and b) since $\Ker(CH^n(V)\to H^{2n}_\et(V,\Z(n)))$ is torsion.
\end{proof}

\begin{rque} If $X$ is smooth projective, Theorem \ref{p1} can be strengthened by replacing $\tilde H^{2n}_\cont$ with $H^{2n}_\cont$, arguing as  in \cite[Th. 8.32]{sheaf}: use a smooth projective spread and Deligne's criterion for the degeneration of the Leray spectral sequence generalising that of \cite[Cor. 3.4]{jann} (see also \cite[Proof of Prop. 7]{tatediv}). I don't know how to obtain such a strengthening for open $X$es.
\end{rque}

Unfortunately, it is not easy to give interesting examples where the conclusions of Theorem \ref{p1} hold unconditionally. The issue is the following: suppose that we start from $X\in B_\Tate(k)$, so that the Tate-Beilinson conjecture is known for $X$, and fibre it over some variety $Y$ with function field $K$. We would like to get a conclusion for the generic fibre $V$ of $X\to Y$. But to pass from $X$ to $V$, we have to remove closed subvarieties $Z$ on which we have no control, even after de Jong desingularisation. If we are in weight $n$, purity will reduce us to the cohomology of (desingularisations of) $Z$ in weights $<n$, but even for $n=2$ this leads to the Tate conjecture in codimension $1$, which is not known in full generality. The best that can be done unconditionally is for $n= 2$ and $\dim X= 2$ \cite[Prop. 4.4]{cell}, but if $Y$ is $1$-dimensional, so will be the generic fibre $V$ and we get an already known statement. At least,  \cite[Cor. 2.5]{cell} gives the Beilinson-Soulé conjecture for function fields $K$ of surfaces of abelian type over $k$. (Consequences of this, like the existence of an abelian category of mixed Tate motives over $K$, do not seem to have been explored.)

\subsection{Characteristic $0$} \label{s5.7} Two fundamental aspects of the picture in characteristic $p$ are:

\begin{itemize}
\item Having the Tate and the Beilinson conjectures jointly allows us to extend conjectural statements from smooth projective varieties to all smooth varieties, because they allow us to strengthen conjectural injectivity/surjectivity to conjectural bijectivity (then allowing for dévissages by purity).
\item The Weil-étale cohomology provides the right conjectural statement in terms of finite generation.
\end{itemize}

 I don't know any conjectural statements which imply the same conclusions as Theorem \ref{p1}  in characteristic $0$. One could try the Tate conjecture over number fields $k$, plus the Bloch-Beilinson conjecture that the $l$-adic Abel-Jacobi map should be injective for all $X\in \bV(k)$ \cite[Lemma 5.6]{beil} (a higher analogue of Beilinson's conjecture over finite fields), but I am not able to use this conjunction to pass from projective to open varieties. It seems that more is needed, like perhaps the right version of Weil-étale cohomology. 
 
 The issue of having the ``right'' conjectures in characteristic $0$ is also discussed in \cite[\S 12]{sheaf}. Hopefully, current research on Weil-étale cohomology in characteristic $0$ (e.g. \cite{morin}) will shed light on this question.

\section{Decomposition of the diagonal}\label{s3}

Let $k$ be a finitely generated field, and let $X$ be a smooth projective $k$-variety.  

\begin{thm}\label{t3} Assume that $CH_0(X_{k(X)})\otimes \Q\by{\deg\otimes \Q} \Q$ is an isomorphism. Let $\delta=|\Coker \deg|$. Then $\Ker \cl^2_X$ is killed by $\delta$, $\Ker \beta^2_X$ is torsion and\\
a) If $\car k=0$,   $CH^2(X)$ is an extension of a finitely generated group by a group of exponent $\delta$. If $k$ is of Kronecker dimension $\le 2$, $\Ker \cl^2_X$ is finite and $CH^2(X)$ is  finitely generated. \\
b) If $\car k=p>0$, the above is true up to $p$-primary groups of finite exponent.
\end{thm}

Recall that the \emph{Kronecker dimension} of $k$ is its transcendence degree over $\F_p$ in characteristic $p$, and $1+$ its transcendence degree over $\Q$ in characteristic $0$.

\begin{proof} The hypothesis means that $X$ has a (rational) decomposition of the diagonal in the sense of Bloch-Srinivas \cite{BS}, i.e. there exists $n>0$ such that $n\Delta_X = \alpha + \beta\in CH^d(X\times X)$ ($d=\dim X$), where $\alpha$ (resp. $\beta$) is supported on $D\times X$ (resp. $X\times V$), $D$ (resp. $V$) being a divisor (resp. a finite number of closed points).

We proceed in four steps:

1) The statements are true with $\delta$ replaced by $n$.

2) The case $X(k)\neq \emptyset$.

3) Refining $n$ to $\delta$ in 1).

4) The case of Kronecker dimension $\le 2$.
\bigskip

1) We apply the technique of Bloch-Srinivas \cite{BS}: by Chow's moving lemma, we reduce to the case where $D\cap V=\emptyset$; if we are in characteristic $0$, we choose a (not necessarily connected) resolution $\tilde D$ of $D$. Since $CH^2(V)=0$, multiplication by $n$ on $CH^2(X)$ factors as
\[CH^2(X)\by{\tilde \alpha^*} CH^1(\tilde D)\by{\rho} CH^2(X)\]
where $\tilde \alpha^*$ is induced by $\alpha$ and $\rho$ is induced by the composition $\tilde D\to D\to X$. This gives the claims, because $CH^1(\tilde D)$ is finitely generated, cycle class maps  are compatible with the action of correspondences \cite[prop. 3.25 and lemma 6.14]{jann} and $\cl^1_{\tilde D}$ is injective as we saw in \S \ref{s1}. 

If we are in characteristic $p>0$, we can proceed similarly by using de Jong's alteration theorem \cite{dJ}, refined by Gabber \cite[Exp. X, th. 2.1]{gabber}, as in \cite[proof of Th. 2.4.2]{ks}: this gives the statement because $l$ is prime to $p$. This implies that $\Ker\beta^2_X$ is torsion, thanks to Lemma \ref{l2}.

For the sequel, we consider the commutative diagram of exact sequences
\begin{equation}\label{eq18}
\Small 
\begin{CD}
&H^3(X,\Z(2))\otimes \Q_l/\Z_l&\to& H^3(X,\Q_l/\Z_l(2))&\to& (CH^2(X)\otimes \Z_l)_\tors&\to 0\\
&@VaV=\cl^{2,3}_X\otimes \Q_l/\Z_lV @VbV\text{inj}V @VcV=\cl^2_X V\\
0\to &H^3_\cont(X,\Z_l(2))\otimes_{\Z_l} \Q_l/\Z_l&\to& H^3_\et(X,\Q_l/\Z_l(2))&\to& H^4_\cont(X,\Z_l(2))
\end{CD}
\end{equation}
where the upper row is  motivic cohomology. The injectivity of $b$ follows from the Merkurjev-Suslin theorem. The snake lemma therefore gives a short exact sequence
\begin{equation}\label{eq17}
0\to \Ker c \to \Coker a \to \Coker b.
\end{equation}

2) Let $x$ be a rational point of $X$. We may choose $V=\{x\}$ and $\beta=n(X\times x)$. In the correspondence ring $CH^d(X\times X)$, the identity $1=\Delta_X$ is the sum of the two idempotents $\pi_0=X\times x$ and $1-\pi_0$, and any module $M$ over this ring decomposes accordingly as a direct sum $M_0\oplus M_+$, where $M_0=\IM \pi_0$ and $M_+=\IM(1-\pi_0)$. Same decomposition for morphisms between such modules. Since $CH^2(X)_0= CH^2(\Spec k) =0$, we have $CH^2(X)=CH^2(X)_+$ and $\Ker c=\Ker c_+$.

\begin{lemma}
 $a_+$ is surjective, and $c_+$ (hence $c$) is injective.
\end{lemma}

\begin{proof}
Since $\alpha=n(1_X-\pi_0)$,  reasoning as in 1) yields a commutative diagram
\[\tiny
\begin{CD}
H^3(X,\Z(2))_+\otimes \Q_l/\Z_l@>{\tilde \alpha^*}>> H^1(\tilde D,\Z(1))\otimes \Q_l/\Z_l@>{\rho}>> H^3(X,\Z(2))_+\otimes \Q_l/\Z_l\\
@Va_+VV @Va'VV @Va_+VV\\
H^3_\cont(X,\Z_l(2))_+\otimes_{\Z_l} \Q_l/\Z_l@>{\tilde \alpha^*}>>H^1_\cont(\tilde D,\Z_l(1))\otimes_{\Z_l} \Q_l/\Z_l@>{\rho}>>H^3_\cont(X,\Z_l(2))_+\otimes_{\Z_l} \Q_l/\Z_l
\end{CD}\]
in which each composition is multiplication by $n$. Here we use that the generalised cycle class maps from motivic cohomology to continuous étale cohomology commute with the action of correspondences, which may be seen most concisely using \eqref{eqchow}.

In particular, the two maps $\rho$ in the diagram are surjective (especially the bottom one), and we reduce to seeing that the middle map $a'$ is surjective. 

In the Milnor exact sequence of \cite[(2.1)]{jann}, the $\lim^1$ disappears, which identifies $H^1_\cont(\tilde D,\Z_l(1))$ with the middle term of the exact sequence
\[0\to \lim \Gamma(\tilde D,\G_m)/l^n\to \lim H^1_\et(\tilde D,\mu_{l^n})\to T_l(CH^1(\tilde D)).\]

But $T_l(CH^1(\tilde D))=0$ since $CH^1(\tilde D)$ is finitely generated, hence $a'$ boils down to the isomorphism
\[\Gamma(\tilde D,\G_m)\otimes \Q_l/\Z_l\iso \widehat{\Gamma(\tilde D,\G_m)}\otimes \Q_l/\Z_l.\]

The injectivity of $c_+$ now follows from the $+$ part of \eqref{eq17}.
\end{proof}

3) Let $x$ be a closed point of $X$. Applying 2) to $X_{k(x)}$, we see by the usual transfer argument that $\Ker\cl^2_X$ is killed by $[k(x):k]$; hence it is killed by $\delta$ which is the gcd of these degrees. On the other hand, we know by 1) (see Lemma \ref{l1} (i)) that $CH^2(X)$ is the direct sum of a free finitely generated subgroup and a group $T$ of finite exponent $n$. By Theorem \ref{t0}, $\cl^2_X(T)$ is finite, thus $T\otimes \Z_l=T\{l\}$ is an extension of a finite group by a group killed by $\delta$; in particular, it is finite for $l\nmid \delta$ (and $0$ for $l\nmid n$). Applying this to all $l\ne p$, we get the promised structure of $CH^2(X)$.

4) Choose a model $\sX$ of $X$, smooth projective over a regular scheme $S$ of finite type over $\Spec \Z[1/l]$ with function field $k$. We have the same diagram as \eqref{eq18} for $\sX$.

We first give the argument in characteristic $p>0$: then $S$ is smooth over $\F_p$. We may choose $S$ such that all the constructions of 1) spread to $\sX$, in particular $\tilde D$ to $\tilde \sD$, $V$ to $\sV$, with $\sD\cap \sV=\emptyset$. We have the theory of Chow correspondences of Deninger-Murre over $S$ \cite{den-murre}; the argument of a) then shows that the finite generation of $CH^2(\sX)$ up to a group of finite exponent would follow from that of $CH^1(\sD)$ and from the vanishing of $CH^2(\sV)$. The first is true by \cite[cor. 1]{picfini}. If $\delta(k)\le 1$, then $CH^2(\sV)=0$ for dimension reasons.  If $\delta(k)=2$, then $\dim S=2$; by \cite[Th. 2]{KS}, $CH^2(T)$ is therefore finite(ly generated) for any $T$ regular and finite over $S$, so up to shrinking $S$ we may achieve again $CH^2(\sV)=0$, and the conclusions of 1) hold for $\sX$. In particular, $\Ker c_\sX$ is torsion of finite exponent (with obvious notation).

Since $H^3_\cont(\sX,\Z_l(2))$ is a finitely generated $\Z_l$-module, it follows from \eqref{eq17} (applied to $\sX$) that $\Ker c_\sX$ is finite.

But $H^4_\cont(\sX,\Z_l(2))\{l\}$ is also finite; it follows that $CH^2(\sX)\{l\}$ is finite. Thus $CH^2(\sX)\otimes\Z_{(l)}$ is a finitely generated $\Z_{(l)}$-module and so is its quotient $CH^2(X)\otimes\Z_{(l)}$.Therefore $CH^2(X)\{l\}=(CH^2(X)\otimes\Z_{(l)})\{l\}$ is finite, and $\Ker \cl^2_X$ is finite.

 In fine, $CH^2(X)_\tors$ is a group of finite exponent whose $l$-primary torsion is finite for any $l\ne p$, hence is finite away from $p$. This concludes the proof of 4) in positive characteristic.
 
 In characteristic $0$, $S$ is flat over $\Spec \Z[1/l]$; up to shrinking, we may assume it to be \emph{smooth}. Thanks to \cite[\S 20.2]{fulton}, the theory of Chow correspondences over a base then goes through as in Deninger-Murre. To conclude the argument, we need to extend \eqref{eq18} and \eqref{eq17} to $\sX$. We use motivic cohomology as in Section \ref{s2.1}; by Theorem \ref{t2.1}, we have indeed $CH^n(\sX)\iso H^{2n}(\sX,\Z(n))$. Moreover, the map $b$ is still injective by \cite[Th. 1.2 2]{geisser3}, so we are done.
\end{proof}

\begin{rques} \label{r1} a) The only essentially new thing in Theorem \ref{t3} is the case of Kronecker dimension $2$: indeed, in characteristic $0$ the finite generation of $CH^2(X)$ goes back to \cite[th. 4.3.10]{colliot} when $\delta=1$ or $k$ is a number field, and the injectivity of $\cl^2_X$ on torsion when $\delta=1$ goes back to \cite[Theorem]{saito} under the weaker hypotheses $H^1(X,\sO_X)=H^2(X,\sO_X)=0$. The proofs given here are different.\\
b) The case of Kronecker dimension $3$ fails because we cannot prove that $CH^2(\sV)=0$ for $S$ small enough: this would follow from the finite generation of $CH^2$ for arithmetic $3$-folds, which is wide open in general.\\ 
c) All that is used in the proof is the existence of a smooth projective variety $\tilde D$ and two correspondences $\tilde \alpha^*$ and $\rho$ (of degrees $-1$ and $+1$) such that $\rho \tilde\alpha^*$ acts on $H^i(X,\Z(2))$ and $H^i_\cont(X,\Z_l(2))$ by multiplication by $n$ for $i=3,4$ and some $n>0$. This implies the conditions of \cite{saito} cited in a); expecting the converse is closely related to Bloch's conjecture on surfaces. This discussion will show up again in Section \ref{s4.5}.\\
d) If we try to relax the hypothesis as in \cite[th. 4.3.10 (i)]{colliot}, assuming only that the birational motive of $X$ (in the sense of \cite{ks}) is a rational direct summand of the motive of a curve $C$, part 1) of the proof goes through because $C$ does not contribute to $CH^2(X)$. But in Part 2), the map $a$ of the diagram is not surjective for $C$.\\
e) It would be nice to get rid of $p$-primary torsion in Theorem \ref{t3} b).
\end{rques}

\section{Counterexamples: codimension $>2$}\label{s4}

In the last three sections, we examine the examples of \cite{sc-suz}, \cite{AS22} and \cite{ct-sc} in the light of the above ideas. They concern torsion classes $x\in CH^n(X)$, for $X$ smooth projective, which are killed by $\cl^n_X$. In this section, $n\ge 3$ and we show as promised in the introduction that one already has $\alpha^n_X(x)=0$, with the notation of \eqref{eq0bis}. 

\subsection{\protect{\cite[Th. 2.3]{sc-suz}}}\label{s4.1} In this reference, a counterexample to the injectivity of $CH^3(X)\allowbreak \otimes \Z_2\to H^6_\cont(X,\Z_2(3))$ is given for a smooth projective variety $X$ over a field $k_0$ of characteristic $\ne 2$, which is a Godeaux-Serre variety associated to a finite $2$-group $G$; the culprit comes from a class $\alpha\in CH^3(BG)$ such that $\alpha_{\bar k_0}\ne 0$ but $\cl^3_{BG}(\alpha)_k=0$ for some finite extension $k$ of $k_0$. In fact,  we already have $\alpha^3_{BG}(\alpha)_{\bar k_0}=0$, hence $\alpha^3_{BG}(\alpha)_k=0$ for suitable $k$ since étale motivic cohomology commutes with filtering limits of schemes (\S \ref{s3.1}: this trivialises the delicate descent argument in the proof of loc. cit., Lemma 2.1 (b)). Indeed, by \cite[Th. 7.1]{BG}, we have
\[H^i_\et(BG_{\bar k_0},\Z(n))\simeq H^i_\et(\bar k_0,\Z(n))\oplus H^i(G,\Z)(n)\]
(where $H^i(G,\Z)$ is the integral cohomology of $G$ and the twist $n$ takes care of the Galois action), so we are reduced to a computation in the cohomology of $G$, which is done in \cite[\S 5]{totaro}. (Recall that $H^6_\et(\bar k_0,\Z(n))$ is uniquely $2$-divisible: \cite[Lemma 7.3 1)]{BG}.)

\subsection{\protect{\cite[\S\S 4 and 5]{sc-suz}}}\label{s4.2} The set up is the same as previously, but the counterexamples rely on the fact that Bloch's map
\[\lambda: CH^3(X_{\bar k_0})\{2\} \to H^5_\et(X_{\bar k_0},\Q_2/\Z_2(3))\]
is not injective. We observe that the Bockstein map 
\[H^5_\et(X_{\bar k_0},\Q_2/\Z_2(3))\to H^6_\cont(X_{\bar k_0},\Z_2(3))\] 
factors through $CH^3_\et(X_{\bar k_0})\otimes \Z_2=H^6_\et(X_{\bar k_0},\Z(3))\otimes \Z_2$.

\begin{qn} The origin of the counterexample in \cite[\S4]{sc-suz} is a nonzero class $x \in H^4_\nr(X,\Q_2/\Z_2(3))$ (loc. cit., Lemma 4.1); by the exact sequence of \cite[p. 998]{cell}, this group surjects onto $\Ker\alpha^3_X$. Can one prove directly that the image of $x$ in $CH^3(X)$ by this map is nonzero? It would explain the counterexample in the spirit of this paper.
\end{qn}

\subsection{\protect{\cite[Th. 1.2 and 1.5]{AS22}}}\label{s4.3} Here $X=S\times E$, where $S$ is a surface and $E$ an elliptic curve; $z=z_1\otimes \tau\in CH^3(X)$ where $z_1\in CH^2(S)$, $\tau\in {}_ l CH^1(E)$ and $\cl^2_S(z_1)$ maps to $0$ in $H^4_\et(S,\mu_l^{\otimes 2})$ (see loc. cit., \S 7). In view of the exact sequence
\[CH^2_\et(S)\by{l} CH^2_\et(S)\to H^4_\et(S,\mu_l^{\otimes 2})\]
we have $\alpha^2_S(z_1)=lt$ for some $t\in H^4_\et(S,\Z(2))$, hence
\[\alpha^3_X(z)= \alpha^2_S(z_1)\otimes \alpha^1_E(\tau)= lt\otimes \alpha^1_E(\tau)=0.\]

\begin{rque}[see \protect{\cite[Rem. 5.4]{ct-sc}}]  For $n>2$ it is easy to produce an example of an $X$ and an element $x\in CH^n(X)$ such that $\alpha^n_X(x)\ne 0$ but $\beta^n_X\alpha^n_X(x)=0$, hence also $\cl^n_X(x)=0$. Namely, start from one of the examples $X_0$ from the next two sections, and let $x_0\in \Ker \cl^2_{X_0}$; just take $X=X_0\times \P^{n-2}$ and $x=x_0\otimes \theta$, where $\theta$ is the canonical generator of $CH^{n-2}(\P^{n-2})=\Z$, and use the projective bundle formula in all theories.
\end{rque}

\section{The counterexample of \protect{\cite[Th. 6.3]{sc-suz}}}\label{s4.4}

Here $X$ is a norm variety of dimension $l^2-1$ over $F=k(t)$ in the sense of \cite{suslin-jou} and \cite{norm}, where $k$ is a global field\footnote{Of characteristic $0$ if $l\ne 2$.}, associated to a nontrivial symbol $s\in H^3_\et(F,\mu_l^{\otimes 2})$. Let $\sR$ be the associated Rost motive: it is a direct summand of the Chow motive of $X$ with coefficients $\Z_{(l)}$, and $CH^2(\sR)\simeq \Z/l$ \cite[proof of Th. 6.3]{sc-suz}. 

Scavia and Suzuki prove that $\cl^2_\sR=0$. We shall recover this result by passing through étale motivic cohomology; namely:

\begin{prop}\label{p2} a) The canonical map
\[CH^2_\et(F)\otimes \Z_{(l)}\to CH^2_\et(\sR)\]
is an isomorphism. \\
b) The map $\beta^2_\sR$ is $0$.
\end{prop}

\begin{proof} a) Morally, this is because $\sR$  splits (i.e. becomes a direct sum of powers of the Lefschetz motive) in the étale topology. This could be given a correct meaning in the étale version of $\DM_\gm^\eff(F)$; we give a more elementary proof, computing the Hochschild-Serre spectral sequence for the weight $2$ étale motivic cohomology of $X$ rather than for its continuous étale cohomology as in \cite{sc-suz}. Namely, this spectral sequence is
\begin{equation}\label{eq11}
E_2^{p,q}(X) = H^p_\et(F,H^q_\et(\bar X,\Z(2)))\Rightarrow  H^{p+q}_\et(X,\Z(2)).
\end{equation}
 
 As in loc. cit., after tensoring with $\Z_{(l)}$ it acquires, as a direct summand, a spectral sequence for the étale motivic cohomology of $\sR$:
 \[E_2^{p,q}(\sR) = H^p_\et(F,H^q_\et(\bar \sR,\Z_{(l)}(2)))\Rightarrow  H^{p+q}_\et(\sR,\Z_{(l)}(2)).\]

Here $\bar \sR= \sR_{\bar F}$. Following the tracks of loc. cit., we use the fact that $\bar \sR\simeq \bigoplus_{i=0}^{j-1} \L^{j(l+1)}$  where $\L$ is the Lefschetz motive \cite[beg. \S 3]{norm}. Since we are over an algebraically closed field, we have $H^q_\et(\L^i,\Z_{(l)}(2)))=0$ if $i>2$; since $l+1\ge 3$, the only summand which contributes is therefore $\un$, which yields $H^q_\et(\bar \sR,\Z_{(l)}(2)))=H^q_\et(\bar F,\Z_{(l)}(2)))$. But \eqref{eq10} shows that this group is uniquely divisible for $q\ne 1$ and that there is an exact sequence
\[0\to \Q_l/\Z_l(2)\to H^1_\et(\bar F,\Z_{(l)}(2))\to H^1_\et(\bar F,\Q(2))\to 0.\]

For $p+q=4$, we therefore have $E_2^{p,q}(\sR)=0$ except for $q=1$, where its value is
\[E_2^{3,1}(\sR)=H^3(F, \Q_l/\Z_l(2)).\]
 
 By unique divisibility and the vanishing of $H^3_\et(\bar F,\Z_{(l)}(2)))$, these are permanent cycles, whence an isomorphism
 $H^3(F, \Q_l/\Z_l(2))\iso CH^2_\et(\sR)$, which is obviously induced by the canonical morphism $\sR\to \un$.
 
 b) This follows from a) by functoriality, since $H^4_\cont(F,\Z_l(2)) = 0$.
 \end{proof}

\begin{rque} As the proof shows, a) is valid for any Rost motive over any field $F$, regardless of its cohomological dimension.
\end{rque}

Let $\alpha$ be a generator of $CH^2(\sR)$; Proposition \ref{p2} a) identifies $\alpha^2_\sR(\alpha)$ with an element of $H^3(F, \Q_l/\Z_l(2))$.  Since it is killed by $l$, it comes by Merkurjev-Suslin from a unique element $t\in H^3(F, \Z/l(2))$. 

\begin{prop}\label{p4} If $l=2$, then $t=s$.
\end{prop}

\begin{proof} Here again, this is valid over any $F$ (of characteristic $\ne 2$). We may assume that $X$ is a Pfister quadric.
Let $K$ be its function field.  Then $\sR$ splits over $K$, hence $CH^2(\sR_K) = 0$ and $t_K = 0$ by functoriality. Since $t \ne 0$, we have $t=s$ by Arason’s theorem \cite[Satz 5.6]{arason}.
\end{proof}

I suppose the same result and proof work for $l$ odd, up to an element of $\F_l^*$, but I am lacking a reference for an analogue of Arason's theorem.

\section{The counterexample of \protect{\cite[Th. 5.3]{ct-sc}}}\label{s4.5} This is by far the most delicate of all. 

Here $X$ is a rational surface over $k=\Q(\sqrt{-p})(t)$, where $p$ is a prime number $\equiv -1\pmod{3}$. More generally, Colliot-Thélène and Scavia consider smooth projective surfaces satisfiying strong conditions  in the spirit of \cite{saito}. Here is a version of their results in this generality, using $CH^2_\et(X)$ and $\beta^2_X$. For simplicity, we invert $p$ if $k$ is of characteristic $p>0$.

We shall use Hypothesis (H4) of \cite[\S 3]{ct-sc}, that we now recall:

\begin{itemize}
\item[(H4)] $H^i(X,\sO_X) = 0$ for $i = 1,2$, $H^3(\bar X^s,\Q_l) = 0$ and, for all $l$, $H^i(X^s,\Z)\{l\} = 0$ if  $i \le 4$. Here $X^s=X\otimes_k k_s$ where $k_s$ is a separable closure of $k$.
\end{itemize}

\begin{rque}
Hypothesis (H4) includes the vanishing of $\Br(\bar X)\{l\}$, or equivalently of the ``transcendental'' part of $H^2_\cont(\bar X,\Q_l)$. Bloch's conjecture then predicts that the Albanese kernel of $\bar X$ vanishes, hence that $CH^2(X)^0$ and $CH^2_\et(X)^0$ are torsion. This is true e.g. for rational surfaces. 
\end{rque}

\begin{thm}\label{t4} Let $X$ be a surface over a field $k$, verifying Hypothesis (H4).\\
a) Let $CH^2_\et(X)^0$ be the kernel of the degree map
\[CH^2_\et(X)\by{f_*} CH^0_\et(\Spec k)=\Z\]
where $f:X\to \Spec k$ is the structural morphism. Let $S$ be the Néron-Severi torus of \cite[p. 6]{ct-sc}. Then there is  a canonical map
\begin{equation}\label{eq20}
CH^2_\et(X)^0\by{\Phi_\et} H^1(k,S)
\end{equation}
such that $\Ker \beta^2_X\subseteq \Ker \Phi_\et\otimes \Z_l$, hence $\Ker \cl^2_X\subseteq \Ker \Phi\otimes \Z_l$ where $\Phi=\Phi_\et\circ \alpha^2_X$.\\
b) If $\cd_l(k)\le 3$, $\Phi_\et$ is surjective. If $\cd_l(k)\le 2$ or if $X$ has a $0$-cycle of degree $1$, $\Ker \Phi_\et$ is uniquely divisible and $\beta^2_X$, $\cl^2_X$ are injective.\\
c) There are compatible exact sequences
\begin{multline}\label{eq16a}
S(k)\to H^3_\et(k,\Q/\Z(2))\\
\to CH^2_\et(X)_\tors\by{\Phi_\et} H^1(k,S)\to H^4_\et(k,\Q/\Z(2))
\end{multline}
\begin{equation}\label{eq16b}
S(k)\to N\\
\to CH^2(X)_\tors\by{\Phi} H^1(k,S)\to H^4_\et(k,\Q/\Z(2))
\end{equation}
where $N=\Ker(H^3_\et(k,\Q/\Z(2))\to H^3_\nr(X,\Q/\Z(2))$ (cf. \cite[(4.1)]{ct-sc}).\\
d) If $\cd_l(k)\le 3$, the inclusions in a) refine to equalities
\[\Ker \beta^2_X= (\Ker \Phi_\et\otimes \Z_l)_\tors, \quad \Ker \cl^2_X= (\Ker \Phi \otimes \Z_l)_\tors\]
 (cf. \cite[Th. 4.8]{ct-sc}). 
\end{thm}

To prove Theorem \ref{t4}, we need the following Propositions \ref{l5} and \ref{p5}.

\begin{prop}\label{l5} Under the hypotheses of Theorem \ref{t4}, the Chow motive $h(X^s)$ has the following integral decomposition:
\[h(X^s) = \un\oplus \NS_X\otimes \L \oplus t_2 \oplus \L^2\]
where (as previously) $\L$ is the Lefschetz motive and $\NS_X$ is the Artin motive associated to the geometric Néron-Severi group $\NS(X^s)$, and $t_2$ represents the Albanese kernel. 
 \end{prop}

\begin{proof} This is the decomposition of \cite[Prop. 14.2.1 and 14.2.3]{kmp}, with two refinements: 1) the summands $h_1(X)$ and $h_3(X)$ vanish because $\Pic^0(X^s)=0$ (Hypothesis (H1) of \cite{ct-sc}), and all Chow-Künneth projectors have integral coefficients. For those of degrees $0$ and $4$, this is obvious since we can choose a rational point. For $\NS_X\otimes \L$ we use the fact that, under \cite[(H1)]{ct-sc}, the intersection pairing on the (torsion-free) Néron-Severi lattice $\NS(X^s)$ is perfect, so the idempotent defining it on \cite[p. 465]{kmp} is integral.
\end{proof}

A problem with the decomposition of Proposition \ref{l5} is that it is not Galois-equivariant in the absence of a $0$-cycle of degree $1$, a crucial condition for the counterexample of \cite{ct-sc} (see loc. cit., Th. 4.2 and part b) of Theorem \ref{t4}). This makes a direct use of the Hochschild-Serre spectral sequence \eqref{eq11} delicate. The most conceptual way to get around this would be to replace it by the finer slice spectral sequence of \cite[(3.2)]{slice}:
\begin{equation}\label{e1.4et}
E_2^{p,q}(X,n)=H^{p-q}_\et(c_q(X),\Z(n-q)_\et)\Rightarrow
H^{p+q}_\et(X,\Z(n))
\end{equation}
but justifying a computation of the slices $c_q(X) \in \DM_\et^\eff(k)$ appears too complicated.\footnote{One should have $c_q(X)=0$  for $q>2$, $c_2(X)=\Z$, $c_1(X)=\NS_X[0]$,  and an exact triangle $t_2 \to c_0(X)\to \Z\to t_2[1]$. One issue is to justify that the motive $t_2$ of Proposition \ref{l5} does define an object of $\DM_\et^\eff(k)$.} Instead, we proceed with a more down-to-earth dévissage as follows:

First, the structural map $X\to \Spec k$ induces a morphism of Chow motives $h(X)\to \un$, hence another one $\L^2\to h(X)$ by Poincaré duality. This induces in turn two morphisms of complexes, with composition $0$:
\begin{equation}\label{eq15}
R\Gamma_\et(k,\Z(2))\to R\Gamma_\et(X,\Z(2))\to R\Gamma_\et(\L^2,\Z(2)).
\end{equation}

Write $R\Gamma_\et(X,\Z(2))^+$ for the homotopy fibre of the second morphism. Taking cohomology, we get a short exact sequence
\[0\to H^4(R\Gamma_\et(X,\Z(2))^+)\to CH^2_\et(X)\to CH^0_\et(k)=\Z\]
which identifies $H^4(R\Gamma_\et(X,\Z(2))^+)$ with the kernel $CH^2_\et(X)^0$ of the degree map.

Then the first morphism of \eqref{eq15} lifts to a morphism $R\Gamma_\et(k,\Z(2))\to R\Gamma_\et(X,\Z(2))^+$ (in the derived category); write $\bar R\Gamma_\et(X,\Z(2))$ for its homotopy cofibre and $\bar H^i_\et(X,\Z(2))$ for the cohomology groups of the latter. This time, we have an exact sequence
\begin{multline}\label{eq16}
\bar H^3_\et(X,\Z(2)))\to H^4_\et(k,\Z(2))\\
\to CH^2_\et(X)^0\by{\rho} \bar H^4_\et(X,\Z(2)))\to H^5_\et(k,\Z(2))
\end{multline}
in which the first (resp. last) term identifies with $H^3_\et(k,\Q/\Z(2))$ (resp. $H^4_\et(k,\Q/\Z(2))$). In particular, $\rho$ is surjective if $\cd_l(k)\le 3$.

\begin{prop}\label{p5} There is an isomorphism $\bar H^3_\et(X,\Z(2))\simeq \ud \oplus S(k)$ and a (split) exact sequence
\begin{equation}\label{eq19}
0\to \ud\to \bar H^4_\et(X,\Z(2))\to H^1(k,S)\to  0
\end{equation}
where $\ud$ stands for uniquely divisible. 
\end{prop}

\begin{proof} We compute the Hochschild-Serre spectral sequence for $\bar R\Gamma_\et(X,\Z(2))$:
\begin{equation}\label{eq13}
E_2^{p,q}=H^p(k,\bar H^q_\et(X^s,\Z(2)))\Rightarrow \bar H^{p+q}_\et(X,\Z(2))
\end{equation}
where $\bar H^*_\et(X^s,\Z(2))$ is defined similarly, for $X^s$. Here, we can use Proposition \ref{l5} to write
\[\bar R\Gamma_\et(X^s,\Z(2)) = R\Gamma_\et(\NS_X\otimes \L \oplus t_2,\Z(2)).\]

We now note that the idempotent in $CH^2(X^s\times X^s)$ defining $\NS_X\otimes \L$ is Galois-invariant by \cite[Lemma 14.2.2]{kmp}, which works for the $\Z$-perfect pairing on $\NS(X^s)$. Accordingly, the direct sum decomposition
\begin{equation}\label{eq14}
\bar R\Gamma_\et(X^s,\Z(2))  = R\Gamma_\et(t_2,\Z(2))\oplus R\Gamma_\et(\NS_X\otimes \L,\Z(2))
\end{equation}
gives a corresponding decomposition of the $E_2$-terms of \eqref{eq13}. For the first summand, we have

\begin{lemma}\label{l6} The groups $H^q_\et(t_2,\Z(2))$ are uniquely $l$-divisible.
\end{lemma}

\begin{proof} One has a decomposition with coefficients $\Z/l^n$
\begin{equation}\label{eq12}\Small
\bar H^q_\et(X^s,\Z/l^n(2))=H^q_\et(\bar k,\Z/l^n(2))\oplus H^q_\et(t_2,\Z/l^n(2))\oplus H^{q-4}_\et(\bar k,\Z/l^n)
\end{equation}
 for any integer $n>0$, which already gives $H^q_\et(t_2,\Z/l^n(2))=0$ for $q\notin [1,3]$. Then, \cite[(H1) and (H3)]{ct-sc} also yield $H^q_\et(t_2,\Z/l^n(2))=0$ for $q=1,3$. Finally, Hypothesis (H1) implies that $\Br(X^s)=0$ and that $H^2_\cont(X^s,\Z_l(2))\to H^2_\et(X^s,\Z/l^n(2))$ is surjective; since the image of the projector defining $t_2$ acting on $H^2_\cont(X^s,\Z_l(2))$ equals $T_l(\Br(X^s))(1)$, we also get $H^2_\et(t_2,\Z/l^n(2))\allowbreak=0$. The unique divisibility follows from the long cohomology exact sequence.
 \end{proof}

Coming back to the proof of Proposition \ref{p5}, the second summand of \eqref{eq14} gives
\[H^q_\et(\NS_X\otimes \L,\Z(2))=\NS_X\otimes H^{q-2}_\et(\bar k,\Z(1))= \NS_X\otimes H^{q-3}_\et(\bar k,\G_m).\]

This is $0$ for $q\ne 3$ and $S(\bar k)$ for $q=3$.

All this allows us to write the $E_2$-terms of \eqref{eq13} as follows:
\[E_2^{p,q} =
\begin{cases}
0&\text{if } p\ne 0 \text{ and } q\ne 3\\
H^p(k,S)&\text{if } p\ne 0 \text{ and } q= 3\\
\ud &\text{if } p= 0 \text{ and } q\ne 3\\
\ud \oplus S(k) &\text{if } p= 0 \text{ and } q= 3.
\end{cases}
\]

This gives the first isomorphism and the exact sequence \eqref{eq19}, except that there is a priori a uniquely divisible term as the cokernel; but it vanishes since $H^1(k,S)$ is torsion.
\end{proof}

\begin{rque} Under Bloch's conjecture,  \cite[Cor. 14.4.9]{kmp} shows that $t_2$ is torsion; then the groups $H^q_\et(t_2,\Z(2))$ of Lemma \ref{l6} all vanish and so do those marked $\ud$ in the $E_2$-terms of the previous proof. In particular, $CH^2(t_2)\subseteq CH^2_\et(t_2)$ vanishes;  applying this over all extensions of $\bar k$, we get $t_2=0$ by a birational Manin identity principle.
\end{rque}

\begin{proof}[Proof of Theorem \ref{t4}]
Composing the map of \eqref{eq19} with the map $\rho$ of \eqref{eq16} we get the map $\Phi_\et$ of a). Thanks to Proposition \ref{p5},  the exact sequence \eqref{eq16} yields \eqref{eq16a}, and \eqref{eq16b} follows from confronting it with \eqref{eq9}. Hence c).

We may now compute the continuous étale cohomology of $X$ in a parallel way to Propositions \ref{l5} and \ref{p5}; for the last part, we get that the cohomology of $t_2$ vanishes and that $H^q_\cont(\NS_X\otimes \L,\Z_l(2))$ is $0$ for $q\ne 2$ and $\NS_X\otimes \Z_l(1)$ for $q=2$, hence with similar notation
\[\bar H^n_\cont(X,\Z_l(2))= H^{n-2}_\cont(k,\NS_X\otimes \Z_l(1)) \]
and the étale cycle class map $\bar H^4_\et(X,\Z(2))\otimes \Z_l\to \bar H^4_\cont(X,\Z_l(2))$ reads as a map
\begin{equation}\label{eq21}
H^1(k,S)\otimes \Z_l\to H^{2}_\cont(k,\NS_X\otimes \Z_l(1))
\end{equation}
which can be interpreted as stemming from the ``Kummer'' exact sequences
\[0\to {}_{l^\nu} S\to S\by{l^\nu} S\to 0\]
because of the canonical isomorphisms ${}_{l^\nu} S\simeq \NS_S/l^\nu$. Hence \eqref{eq21} is \emph{injective} by \cite[Prop. 2.2 c)]{ct-sc}, since $H^1(k,S)$ is a torsion group. Thus we get a commutative diagram
\[\begin{CD}
CH^2_\et(X)^0\otimes \Z_l@>{\Phi_\et\otimes \Z_l}>> H^1(k,S)\otimes \Z_l\\
@V(\cl^2_X)^0VV @VV\text{inj}V\\
H^4_\cont(X,\Z_l(2))@>>> H^{2}_\cont(k,\NS_X\otimes \Z_l(1))
\end{CD}\]
which shows that $\Ker \cl^2_X=\Ker (\cl^2_X)^0\subseteq \Ker \Phi_\et$, completing the proof of a) and yielding  the assertions of b) (use Theorem \ref{t3} for the vanishing of $\Ker\beta^2_X$ and $\Ker\cl^2_X$).

For d), the point is simply that the commutative square above extends to a larger commutative diagram of exact sequences
\[\begin{CD}
CH^2_\et(k)\otimes \Z_l @>>> CH^2_\et(X)^0\otimes \Z_l@>{\Phi_\et\otimes \Z_l}>> H^1(k,S)\otimes \Z_l\\
@VVV @V(\cl^2_X)^0VV @VV\text{inj}V\\
H^4_\cont(k,\Z_l(2))@>>> H^4_\cont(X,\Z_l(2))@>>> H^{2}_\cont(k,\NS_X\otimes \Z_l(1))
\end{CD}\]
and that the bottom left term is $0$.
\end{proof}

\begin{rques} 1) In order to get back the counterexample of \cite[Th. 5.3]{ct-sc}, it would remain to see that the map $S(k)\to H^3_\et(k,\Q/\Z(2))$ of \eqref{eq16b} agrees with that of \cite[(4.1)]{ct-sc} (presumably the two sequences coincide).\\ 
2) In \cite[Th. 4.8/Th. A.1]{ct-sc}, an analogue of Theorem \ref{t4} d) is stated for $\Ker \Phi$ without assuming $\cd_l(k)\le 3$. \end{rques}


\begin{thebibliography}{II}
\bibitem{AS22} Th. Alexandrou, S. Schreieder, {\it On Bloch’s map for torsion cycles over non-closed fields}, Forum Math., Sigma {\bf 11} (2023),  e53, 21 pp. 
\bibitem{arason} J. Kr. Arason {\it Cohomologische Invarianten quadratischer Formen},
J. Alg. {\bf 36} (1975), 448--491.
\bibitem{beil} A. Beilinson {\it Height pairing between algebraic cycles}, {\it in} $K$-theory, arithmetic and geometry (Yu. Manin, ed.), Lect. Notes in Math. {\bf 1289}, Springer, 1987, 1--26.
\bibitem{BS} S. Bloch, V. Srinivas {\it Remarks on correspondences and algebraic cycles}, Amer. J. Math. {\bf 105}, No. 5 (1983), 1235--1253.
\bibitem{bloch-adv} S. Bloch {\it Algebraic cycles and higher $K$-theory}, Adv. in Math. {\bf 61} (1986), 267--304.
\bibitem{bloch-corr} S. Bloch {\it The moving lemma for higher Chow groups}, J. Alg. Geom. {\bf 3} (1994), 537--568.
\bibitem{colliot} J.-L. Colliot-Thélène {\it Birational invariants, purity and the Gersten conjecture}, {\it in} $K$-theory and algebraic geometry: connections with quadratic forms and division algebras (Santa Barbara, CA, 1992), 1--64, 
Proc. Sympos. Pure Math., {\bf 58}, Part 1, Amer. Math. Soc., Providence, RI, 1995.
\bibitem{ct-sc} J.-L. Colliot-Thélène, F. Scavia {\it Sur l'injectivité de l'application cycle de Jannsen}, preprint, 2022, \url{https://arxiv.org/abs/2212.05761}, to appear in `Perspectives on four decades: Algebraic Geometry 1980 - 2020. In memory of Alberto Collino'.
\bibitem{den-murre} C. Deninger, J. P. Murre: Motivic decomposition of abelian schemes and the Fourier transform, {\it J. reine angew. Math.} {\bf 422} (1991), 201--219.
\bibitem{fulton} W. Fulton Intersection theory, Springer, 1984.
\bibitem{geisser} T. Geisser {\it Tate's conjecture, algebraic cycles and rational $K$-theory in characteristic $p$}, $K$-theory {\bf 13} (1998), 109--122.
\bibitem{geisser3} T. Geisser {\it Motivic cohomology over Dedekind rings},  Math. Z. {\bf 248} (2004), 773--794. 
\bibitem{geisser2} T. Geisser {\it Weil-étale cohomology over finite fields},  Math. Ann. {\bf 330} (2004), 665--692.
\bibitem{sga4} A. Grothendieck {\it Site et topos étale d'un schéma}, {\it in} Théor!e des topos et cohomologie étale des schémas (SGA 4), Tomor 2, Lect. Notes in Math. {\bf 270}, Springer, 1972, 341--365.
\bibitem{slice} A. Huber, B. Kahn {\it The slice filtration and mixed Tate motives}, Compositio Math. {\bf 142} (2006), 907--936.
\bibitem{gabber} L. Illusie, M. Temkin Travaux de Gabber sur l’uniformisation locale et la cohomologie étale des schémas quasi-excellents, L. Illusie et al., eds., Astérisque {\bf 363--364}, Société Mathématique de France, Paris, 2014.
\bibitem{jann} U. Jannsen {\it Continuous étale cohomology}, Math. Ann. {\bf 280} (1988), 207--245.
\bibitem{jann2} U. Jannsen Mixed motives and algebraic $K$-theory, Lect. Notes in Math. {\bf 1400}, Springer, 1990.
\bibitem{jannsen} U. Jannsen {\it Motives, numerical equivalence, and semi-simplicity}, Invent. Math. {\bf 107} (1992), 447--452.
\bibitem{jann3} U. Jannsen {\it Motivic sheaves and filtrations on Chow groups}, {\it in} Motives, Proc. Symp. pure Math. {\bf 55} (I), 245--302, AMS, 1994.
\bibitem{dJ} A. J. de Jong {\it Smoothness, semi-stability and alterations}, Publ. Math. IH\'ES {\bf 83} (1996), 51--93.
\bibitem{sheaf} B. Kahn {\it A sheaf-theoretic reformulation of the Tate conjecture}, preprint, 1998, \url{https://arxiv.org/abs/math/9801017}.
\bibitem{glr} B. Kahn {\it The Geisser-Levine method revisited and algebraic cycles over a finite field}, Math. Ann. {\bf 324} (2002), 581--617.
\bibitem{cell} B. Kahn {\it Équivalences rationnelle et numérique sur certaines variétés de type abélien sur un corps fini}, Ann. Sci. Ec. Norm. Sup. {\bf 36} (2003), 977--2002.
\bibitem{picfini} B. Kahn {\it Sur le groupe des classes d'un schéma arithmétique}, Bull. Soc. math. France
{\bf 134} (2006), 395--415.
\bibitem{ihes} B. Kahn {\it The full faithfulness conjectures in characteristic $p$}, {\it in} ``Autour des motifs'', École d'été franco-asiatique de géométrie algébrique et de théorie des nombres (IHÉS, juillet 2006), Vol. II, Panoramas et Synthèses {\bf 41}, SMF, 2014, 187--244.
\bibitem{cycletale} B. Kahn {\it Classes de cycles motiviques étales}, Alg. Numb. Theory {\bf 6-7} (2012), 1369--1407.
\bibitem{tatediv} B. Kahn {\it Sur la conjecture de Tate pour les diviseurs}, Ess. Numb. Theory {\bf 2} (2023), 83--92..
\bibitem{kmp} B. Kahn,  J.P. Murre, C. Pedrini {\it On the transcendental part of the motive of a surface}, {\it in} Algebraic cycles and motives (for J.P. Murre's 75th birthday), LMS Series {\bf 344} (2), Cambridge Univ. Press, 2007, 143--202.
\bibitem{BG} B. Kahn, T. K. N. Nguyen {\it Sur l’espace classifiant d’un groupe algébrique linéaire, I}, J. Math. Pures Appl. {\bf 102} (2014) 972--1013. 
\bibitem{ks} B. Kahn, R. Sujatha {\it Birational motives, I: pure birational motives}, Ann. $K$-theory {\bf 1} (2016), 379--440.
\bibitem{norm} N. A. Karpenko, A. S. Merkurjev {\it On standard norm varieties}, Ann. Sci. Éc. Norm. Sup. {\bf 46} (2013), 175--214.
\bibitem{KS} K. Kato, S. Saito {\it Unramified class field theory for arithmetic surfaces}, Ann. of Math. {\bf 118} (1983), 241--275.
\bibitem{kim} S.-I. Kimura {\it Chow groups are finite dimensional, in some sense}, Math. Ann. {\bf 331} (2005), 173--201.
\bibitem{levine} M. Levine {\it Techniques of localization in the theory of algebraic cycles}, 
J. Alg. Geom. {\bf 10} (2001), 299--363. 
\bibitem{licht} S. Lichtenbaum {\it The Weil-étale topology on schemes over finite fields}, Compos. Math. {\bf 141} (2005),  689--702.
\bibitem{milned} J. S. Milne Arithlmetic duality theorems, Second edition. BookSurge, LLC, Charleston, SC, 2006.
\bibitem{milne} J. S. Milne {\it Motives over finite fields}, {\it in} Motives (Seattle, WA, 1991), Proc. Sympos. Pure Math., {\bf 55}, Part 1  (1994), Amer. Math. Soc., Providence, RI, 401--459.
\bibitem{morin} B. Morin {\it Topological Hochschild homology and Zeta-values}, preprint, 2020, \url{https://arxiv.org/abs/2011.11549}.
\bibitem{ns} Yu. P. Nesterenko, A. A. Suslin {\it Homology of the full linear group over a local ring, and Milnor's $K$-theory}, Izv. Akad. Nauk SSSR {\bf 53} (1989); Engl. transl.: Math. USSR Izv. {\bf 34} (1990), 121--145. 
\bibitem{rost} M. Rost {\it Chow groups with coefficients}, Doc. Math. {\bf 1} (1996) 319--393.
\bibitem{saito} S. Saito {\it On the cycle map for torsion algebraic cycles of codimension two}, Invent. Math. {\bf 106} (1991), 443--460.
\bibitem{sc-suz} F. Scavia, F. Suzuki {\it Non-injectivity of the cycle class map in continuous $l$-adic cohomology}, Forum Math. Sigma {\bf 11}  (2023), e6, 19 pp.
\bibitem{soule} C. Soul\'e {\it Groupes de Chow et $K$-th\'eorie de vari\'et\'es sur un corps fini}, Math. Ann. {\bf 268} (1984), 317--345.
\bibitem{suslin-jou} A. Suslin and S. Joukhovitski {\it Norm varieties}, J. Pure Appl. Algebra {\bf 206} (2006), 245--276.
\bibitem{tate} J. Tate {\it Conjectures on algebraic cycles in $l$-adic cohomology}, {\it in} Motives, Proc. Symposia Pure Math. {\bf 55} (1), AMS, 1994, 71--83.
\bibitem{totaromilnor} B. Totaro {\it Milnor's $K$-theory is the simplest part of algebraic $K$-theory}, $K$-theory
\bibitem{totaro} B. Totaro {\it Torsion algebraic cycles and complex cobordism}, J. Amer. Math. Soc. {\bf 10} (1997), 467--493.
\bibitem{voetri} V. Voevodsky {\it Triangulated categories of motives over a field}, {\it in} Cycles, transfers and motivic cohomology theories, Ann. of Math. Studies {\bf 143}, 2000, 188--238.
\end{thebibliography}
\end{document}